\documentclass{amsart}
\usepackage{amsmath}
  \usepackage{paralist}
  \usepackage{graphics} 
  \usepackage{epsfig} 
\usepackage{graphicx}  
 \usepackage[colorlinks=true]{hyperref}
 \usepackage{mathabx} 
\hypersetup{urlcolor=blue, citecolor=red}
\newtheorem{theorem}{Theorem}[section] 
\newtheorem{theorem*}{Theorem}[]
\theoremstyle{definition}
\newtheorem{corollary}[theorem]{Corollary} 
\newtheorem{lemma}[theorem]{Lemma}
\newtheorem{proposition}[theorem]{Proposition}
\newtheorem{definition}[theorem]{Definition}
\newtheorem{remark}[theorem] {Remark}
\numberwithin{equation}{section} 
\newtheoremstyle{cited}
  {3pt}
  {3pt}
  {\itshape}
  {}
  {\bfseries}
  {.}
  {.3em}
  {\thmname{#1} \thmnumber{\bf{#2}} \thmnote{\normalfont#3}}
\theoremstyle{cited}
\newtheorem*{thm*}{Theorem}

\newcommand{\Rn}{\ensuremath{\mathbb{R}^n}}
\newcommand{\R}{\ensuremath{\mathbb{R}}}
\newcommand{\N}{\ensuremath{\mathbb{N}}}
\newcommand{\obal}{\ensuremath{\mathbb{R}^n} \setminus B_r}

\newcommand{\frlap}{\ensuremath{(-\Delta)^s}}
\newcommand{\eee}{\ensuremath{\varepsilon}}

\newcommand{\F}{\ensuremath{\mathcal{F}}}
\newcommand{\Sa}{\ensuremath{\mathcal{S}}}
\newcommand{\wck}{\ensuremath{\widecheck}}
\renewcommand\footnotemark{}

\title[Green function for the fractional Laplacian]{Some observations on the Green function for the ball in the fractional Laplace framework}
\author {Claudia Bucur}
\email{claudia.bucur@unimi.it}

\begin{document}
\maketitle

\centerline{\scshape Claudia Bucur}
\medskip
{\footnotesize 
 \centerline{Dipartimento di Matematica ``Federigo Enriques''}
\centerline{Universit\`a degli Studi di Milano} 
   \centerline{Via Cesare Saldini, 50, I-20133, Milano, Italy}
 } 

\bigskip

 \begin{abstract}
We consider a fractional Laplace equation and we give a self-contained elementary exposition of the representation formula for the Green function on the ball. In this exposition, only elementary calculus techniques will be used, in particular, no probabilistic methods or computer assisted algebraic manipulations are needed.  The main result in itself is not new (see for instance \cite{conto,stableprocess}), however we believe that the exposition is original and easy to follow, hence we hope that this paper will be accessible to a wide audience of young researchers and graduate students that want to approach the subject, and even to professors that would like to present a complete proof in a PhD or Master Degree course.
\end{abstract}
\tableofcontents

\section{Introduction}

The Green function for the ball in a fractional Laplace framework naturally arises in the study of the representation formulas for the fractional Laplace equations. In particular, in analogy to the classical case of the Laplacian, given an equation with a known forcing term on the ball and vanishing Dirichlet data outside the ball, the representation formula for the solution is precisely the convolution of the Green function with the forcing term. As in the classical case, the Green function is introduced in terms of the Poisson kernel. For this, we will provide both the representation formulas for the problems
	\begin{equation}\label{LaplaceeqD}
		\begin{cases}
			\frlap u= 0   \qquad  &\mbox{ in }  {B_r},    \\
			 u= g \qquad  &\mbox{ in } {\obal}
		\end{cases}
	\end{equation}
and
	\begin{equation}\label{PoissoneqD}
		\begin{cases}
   			\frlap u= g \qquad  &\mbox{ in }    {B_r},    \\
			 u= 0   \qquad  &\mbox{ in }  {\obal}
	\end{cases}
	\end{equation}
in terms of the fractional Poisson kernel and respectively the Green function. Moreover, we will prove an explicit formula for the Green function on the ball.

Here follow some notations and a few preliminary notions on the fractional Laplace operator and on the four kernels that play particular roles in our study: the $s$-mean kernel, the fundamental solution, the Poisson kernel and the Green function.

We briefly introduce the Schwartz space (refer to \cite{Schwartz} for details) as
 the functional space
	\[ \mathcal{S} ({\Rn}): = \left \{  f \in C^\infty(\Rn)  \text{ s.t. }    \forall \alpha, \beta \in \mathbf{N}^n_0, \,  \sup_{x\in {\Rn}}|x^{\beta} D^{\alpha} f(x)|< + \infty  \right \}. \]
In other words, the Schwartz space consists of smooth functions whose derivatives (including the function itself) decay at infinity faster than any power of $x$; we say, for short, that Schwartz functions are rapidly decreasing. Endowed with the family of seminorms
	\begin{equation} \label{seminormss} [f]_{\Sa(\Rn)}^{\alpha,N} =\sup_{x\in {\Rn}} (1+|x|)^N  \sum_{|\alpha| \leq N}| D^{\alpha} f(x)|, \end{equation}
	the Schwartz space is a locally convex topological space. We denote by $\Sa'(\Rn)$ the space of tempered distributions, the topological dual of $\Sa(\Rn)$.

We set the following notations for the Fourier and the inverse Fourier transform (see for instance, \cite{Fourier} for details)
		\begin{equation*}
		\widehat f(\xi) = \mathcal{F} f(\xi):=  \int_{\Rn} f(x) e^{- {2\pi}  i x \cdot \xi} \, dx
	\end{equation*}
	{respectively}
	\begin{equation*}	
		 \wck {f}(x)=  \mathcal{F}^{-1}  f(x)  = \int_{\Rn}  f (\xi) e^{ {2\pi} i x \cdot \xi} \, d\xi.
	\end{equation*}
Here the original space variable is denoted by $x \in \Rn$ and the frequency variable by $\xi \in \Rn$. We recall that the Fourier and the inverse transform are well defined for $f\in L^1(\Rn)$, whereas $f(x)= \F \big(\F^{-1} f)(x) =  \F^{-1} \big(\F f)(x)$ almost everywhere if both $f$ and $\widehat f\in L^1(\Rn)$, and pointwise if $f$ is also continuous.
Also for all $ f,\,g \in L^1(\Rn)$
	\[  \int_{\Rn} \widehat f(\xi)  g(\xi) \, d\xi = \int_{\Rn}f(\xi) \widehat g (\xi) \, d\xi  . \]
 The pointwise product is taken into the convolution product and vice versa, namely for all $ f,\,g \in L^1(\Rn)$
	\[  \mathcal{F} (f*g)=  \mathcal{F}(f) \, \mathcal{F}(g) . \]
	On the Schwartz space, the Fourier transform gives a continuous bijection between $\mathcal{S}(\Rn)$  and $\mathcal{S}(\Rn)$.
We say that $\widehat{f}$ is the Fourier transform of $f$ in a distributional sense, for $f$ that satisfies $ \int_{\Rn} \frac{|f(x)|}{1+|x|^{p}}\, dx <\infty $ for some $p \in \mathbb{N}$ if, for any $\varphi \in \mathcal{S}(\Rn)$ we have that
	\begin{equation} \label{dists1} \int_{\Rn} \widehat f(x) \varphi(x) \, dx= \int_{\Rn} f(x) \widehat \varphi (x) \, dx.\end{equation} We remark that the integral notation is used in a formal manner whenever the arguments are not integrable.

Let $s\in(0,1)$ be fixed. We introduce the fractional Laplacian for $u$ belonging to the Schwartz space.
\begin{definition} The fractional Laplacian of $u \in \mathcal{S} ({\Rn})$ is defined as
	\begin{equation}\label{frlapdef}
		\begin{split}
		 \frlap u(x)\; :=\; & C (n,s) P.V. \int_{\Rn} \frac{ u(x)- u(y)}{|x-y|^{n+2s}}  \, dy \\
					\;=\; & C(n,s) \lim_{\eee \to 0} \int_{{\Rn}\setminus B_{\eee}(x)} \frac{u(x) - u(y)}{|x - y|^{n+2s}}\, dy,
		\end{split}
	 \end{equation}
where $C(n,s)$ is a constant depending only on $n$ and $s$.
\end{definition}
Here, \emph{P.V.} is a commonly used abbreviation for "in the principal value sense" (as defined by the latter equation). The dimensional constant $C(n, s)$ is given in \cite{galattica}, formula 3.2, as :
	\begin{equation} \label{GCNS}
		C(n, s) :=   \bigg(\int_{\Rn} \frac{1 - \cos(\eta_1)}{|\eta|^{n+2s}} \, d\eta\bigg)^{-1},
	\end{equation}
where $\eta_1$ is, up to rotations, the first coordinate of $\eta\in \Rn$.
\begin{remark}\label{blacaz}
 The definition \eqref{frlapdef} is well posed for smooth functions belonging to a weighted $L^1$ space, that we define as follows. For $s\in (0,1)$
	\[L^1_s (\Rn):=\Big\{ u \in L^1_{\text{loc}}(\Rn)\; \mbox{ s.t. } \; \int_{\Rn}\frac{ |u(x)|}{1+|x|^{n+2s}} \, dx <\infty\Big\},\]
	endowed naturally with the norm
		\[ \| u\|_{L^1_s({\Rn})} := \int_{\Rn}\frac{ |u(x)|}{1+|x|^{n+2s}} \, dx.\]
Let $\eee$ be positive, sufficiently small. Then indeed, for $u \in L_s^1(\Rn)$ and $C^{0,2s+\eee}$ (or $C^{1,2s+\eee-1}$ for $s\geq 1/2$) in a neighborhood of $x\in \Rn$, the fractional Laplacian is well defined in $x$ as in \eqref{frlapdef}. See for the proof Proposition 2.1.4 in \cite{Silvestre}, where an approximation with Schwartz functions is performed. 	We often use this type of regularity to obtain pointwise solutions to the posed equations. We will always write $C^{2s+\eee}$ to denote both $C^{0,2s+\eee}$ for $s<1/2$  and $C^{1,2s+\eee-1}$ for $s\geq 1/2$.
\end{remark}


In definition \eqref{frlapdef}, the singular integral can be substituted with a weighted second order differential quotient by performing the changes of variables $\tilde y=x+y$ and $\tilde y=x-y$ and summing up. In this way, the singularity at the origin can be removed, as we observe in the following lemma (see for the proof Section 3.1 in \cite{nonlocal}).

\begin{lemma}
Let $\frlap$  be the fractional Laplace operator defined by \eqref{frlapdef}. Then for any smooth $u$
	\begin{equation} \frlap u(x)= \frac{C(n,s)}{2} \int_{\Rn} \frac{2 u(x) -u(x+y)-u(x-y)} {|y|^{n+2s} } \, dy. \label{frlap2def} \end{equation}
\end{lemma}

For $u \in \mathcal{S}(\Rn)$ the fractional Laplace operator can also be expressed by means of the Fourier transform, according to the following lemma (see Proposition 3.3 in \cite{galattica} for the proof).
\begin{lemma}
Let $\frlap$  be the fractional Laplace operator defined by \eqref{frlapdef}. Then for any $u \in \mathcal{S}(\Rn)$
	\begin{equation} \frlap u(x)=\mathcal{F}^{-1} \Big( \big({2\pi}|\xi|\big)^{2s} \widehat u (\xi)\Big) .\label{frlaphdef} \end{equation}
\end{lemma}

We refer usually to pointwise solutions, nevertheless distributional solution will also be employed. Following the approach in \cite{Silvestre} (see Definition 2.1.3),  we introduce a suitable functional space where distributional solutions can be defined.
 Let \[ \Sa_s(\Rn) := \Big\{ f \in C^{\infty}(\Rn) \mbox{ s.t. } \forall \alpha \in \mathbf{N}^n_0, \; \sup_{x\in \Rn} \big(1+|x|^{n+2s}\big) |D^{\alpha} f(x) | <+\infty \Big\}. \]   The linear space $\Sa_s(\Rn)$ endowed with the family of seminorms
 	\begin{equation*} \label{seminormss1}[f]^{\alpha}_{\Sa_s(\Rn)}:=\sup_{x\in \Rn} \big(1+|x|^{n+2s}\big) |D^{\alpha} f (x)|\end{equation*} is a locally convex topological space. We denote with $\Sa_s'(\Rn)$ the topological dual of $\Sa_s(\Rn)$.
 	
We notice that if $\varphi \in \Sa(\Rn)$ then $\frlap \varphi \in \Sa_s(\Rn)$, which makes this framework appropriate for the distributional formulation. In order to prove this, we observe that for any $x\in \Rn \setminus B_1$ the bound
 	\begin{equation} |\frlap \varphi(x)| \leq c_{n,s} |x|^{-n-2s}\label{frb1} \end{equation}
 	follows from the upcoming computation and the fact that $\varphi \in S(\Rn)$
 	\[\begin{split}
 	&\;  |\frlap \varphi (x) |\\ \leq \;&  \int_ {B_{\frac{|x|}{2}}} \frac{\big|2 \varphi (x) -\varphi(x-y) - \varphi(x+y) \big| }{|y|^{n+2s}}\, dy
 	 + 2 \int_ {\Rn \setminus B_{\frac{|x|}{2}}} \frac{\big| \varphi (x)-  \varphi(x+y) \big|}{|y|^{n+2s}}\, dy  \\
 	 \leq &\; c_{n,s} |x|^{-n-2s} \bigg( \sup_{z\in \Rn} (1+|z|)^{n+2}|D^2\varphi (z)|   +  \sup_{z\in \Rn} (1+|z|)^{n}|\varphi (z)|
 	 + \|\varphi \|_{L^1(\Rn)}\bigg)  .
 	\end{split}\]
 	Moreover, we observe that, up to constants,
 		\[ \begin{split}   \partial_{x_i} \frlap \varphi(x) &\;=\partial_{x_i} \F^{-1}\Big(|\xi|^{2s}\widehat \varphi(\xi)\Big)(x)  		
 		= \F^{-1}\Big( i\xi_i|\xi|^{2s}\widehat \varphi(\xi) \Big)(x)\\
 		&\;= \F^{-1}\Big(|\xi|^{2s}\widehat{\partial_{x_i} \varphi} (\xi) \Big) (x)= \frlap \partial_{x_i} \varphi(x).
 			\end{split}\]
Hence, by iterating the presented argument, one proves that $\frlap \varphi \in \Sa_s(\Rn)$, which assures our claim. Then we have the following definition.
\begin{definition}
Let $f \in \Sa'(\Rn)$, we say that $u \in \Sa_s'(\Rn)$ is a distributional solution of
		\[ \frlap u=f \; \mbox{in } \; \Rn\] if
		\begin{equation}\label{disf1} <u,\frlap \varphi>_s= \int_{\Rn} f(x)\varphi(x) \,dx\quad \mbox{for any} \; \varphi \in \Sa(\Rn), \end{equation}
where $<\cdot, \cdot>_s$ denotes the duality pairing of $\Sa_s'(\Rn)$ and $\Sa_s(\Rn)$ and the latter (formal) integral notation designates the pairing $\Sa(\Rn)$ and $\Sa'(\Rn)$.
\end{definition}
We also use the integral notation for the pairing $<\cdot,\cdot>_s$ in a purely formal manner whenever the arguments are not integrable. We notice that the inclusion $L_s^1(\Rn) \subset \Sa_s'(\Rn)$ holds, in particular for any $u\in L_s^1(\Rn)$ and $\psi \in \Sa_s(\Rn)$ we have
	\begin{equation}\begin{split}
	\Big| < u ,\psi>_s  \Big| \leq &\;  \int_{\Rn} |u(x)|\, |\psi(x)|\, dx \leq \int_{\Rn} \frac{|u(x)|}{1+|x|^{n+2s}}  (1+|x|^{n+2s})|\psi(x)|\, dx \\
	\leq&\;  [\psi]^0_{\Sa_s(\Rn)} \| u\|_{L_s^1(\Rn)}.
	\label{db1} \end{split}
	\end{equation}

We introduce now the four functions $A_r$, $\Phi$, $P_r$ and $G$, namely the $s$-mean kernel, the fundamental solution, the Poisson kernel and the Green function. The reader can see Section 2.2 in \cite{PDE} for the theory in the classical case.

\begin{definition}
 Let  $r>0$ be fixed. The function $A_r$ is defined by
	\begin{equation} \label{smeandefn}
	A_r(y) :=  \begin{cases}
		c(n,s)  \displaystyle \frac{r^{2s}}{(|y|^2-r^2)^s|y|^n} \quad &y \in \Rn \setminus \overline{B_r},\\
		0 \quad &y \in \overline B_r,
		\end{cases}
	\end{equation}
where $ c(n,s)$ is a constant depending only on $n$ and $s$.
\end{definition}

\begin{definition} For any $x\in \Rn \setminus \{0\}$ the function $\Phi$ is defined  by
   	\begin{equation}
		\Phi(x) :=
			\begin{cases}
			 \displaystyle a(n,s){|x|^{-n+2s}} \quad &\text {if } n \neq 2s, \\
			\displaystyle a\Big(1,\frac{1}{2}\Big) \log |x| \quad & \text {if } n = 2s  ,
			\end{cases}
	\label{fundsolution}
	\end{equation}
where $a(n,s)$ is a constant depending only on $n$ and $s$.
\end{definition}

\begin{definition}
Let $r>0$ be fixed. For any $ x \in B_r$ and any $ y \in \Rn \setminus \overline{B}_r$, the Poisson kernel $P_r$ is defined by
	\begin{equation}
	 P_r(y,x) :=  c(n,s) \Bigg (\frac {r^2-|x|^2}{|y|^2-r^2}\Bigg)^s \frac {1}{|x-y|^n}. \label{poissondefn}
	\end{equation}
\end{definition}

The Poisson kernel is used to build a function which is known outside the ball and $s$-harmonic in $B_r$. Indeed, it is used to give the representation formula for problem \eqref{LaplaceeqD},  as stated in the following theorem.

 \begin{thm*} [\ref{theorem:DPL}]
 Let $r>0$, $g \in L^1_s(\Rn) \cap C({\Rn})$ and let
\begin{equation}
		 u_g(x) : =
			\begin{cases}	
				\displaystyle  \int_{{\Rn}\setminus B_r} P_r(y,x) g(y)\, dy &\quad  \, \text{if } x\in B_r, \\
				g(x) &\quad \, \text{if } x \in {\obal}.
			\end{cases} \label{solD}
	\end{equation}
Then $u_g$ is the unique pointwise continuous solution of the problem \eqref{LaplaceeqD}
	\begin{equation*}
	\begin{cases}
	\frlap u= 0 \qquad  &\mbox{ in }  {B_r},
\\	u= g \qquad  &\mbox{ in }  {\obal}.
		\end{cases}
	\end{equation*}
\end{thm*}

\begin{definition} Let $r>0$ be fixed. For any $x, z \in B_r$ and $x\neq z$, the function $G$ is defined by
\begin{equation} G(x,z) := \Phi(x-z) -\int_{\obal} \Phi(z-y) P_r(y,x) \, dy . \label{greendefn}\end{equation}
\end{definition}

From this definition, a formula that is more suitable for applications can be deduced. Indeed, one of the main goals is to prove this simpler, explicit formula for the Green function on the ball, by means of elementary calculus techniques. At this purpose, Theorem \ref{theorem:thm1} establishes a symmetrical expression for~$G$.
\begin{thm*} [\ref{theorem:thm1}]
Let $r>0$ be fixed and let $G$ be the function defined in \eqref{greendefn}. Then if $n \neq 2s$\begin{equation} \label{forgkns}   G(x,z) =  \kappa(n,s) |z-x|^{2s-n}  \int_0^{r_0(x,z)}  \frac{t^{s-1}} {(t+1)^\frac{n}{2}} \, dt ,\end{equation}
where
 	\begin{equation} \displaystyle r_0(x,z) = \frac{(r^2-|x|^2)(r^2-|z|^2)}{r^2|x-z|^2}\label{ro} \end{equation}
and $\kappa(n,s)$ is a constant depending only on $n$ and $s$.\\
For $n=2s$, the following holds
\begin{equation} G(x,z)= \kappa\Big(1,\frac{1}{2}\Big) \log\bigg( \frac{r^2-xz+\sqrt{(r^2-x^2)(r^2-z^2)}}{r|z-x|}\bigg).\label{formn1s12} \end{equation}
\end{thm*}
This result is not new (see \cite{conto,stableprocess}), however, the proof we provide uses only calculus techniques, therefore we hope it will be accessible to a wide audience. It makes elementary use of special functions like the Euler-Gamma function, the Beta and the hypergeometric function, that are introduced in the Appendix (see the book \cite{gamma} for details). Moreover, the point inversion transformations and some basic calculus facts, that are also outlined in the Appendix, are used in the course of this proof.

The main property of the Green function is stated in the upcoming Theorem~\ref{theorem:thm2}. The function $G$ is used to build the solution of an equation with a given forcing term in a ball and vanishing Dirichlet data outside the ball, by convolution with the forcing term. While this convolution property in itself may be easily guessed from the superposition effect induced by the linear character of the equation, the main property that we point out is that the convolution kernel is explicitly given by the function $G$.

\begin{thm*} [\ref{theorem:thm2}]
Let $r>0$, $h \in {C^{2s+\eee}(  B_r)  \cap C(\overline B_r)}$ and let
	 \begin{equation*}
		u(x) : =
			\begin{cases}
		\displaystyle \int_{B_r} h(y) G(x,y) \, dy \quad & \text{ if } x\in B_r, \\
		0 \quad \quad & \text{ if } x \in {\obal}.
		\end{cases}
	\end{equation*}
Then $u$ is the unique pointwise continuous  solution of the problem \eqref{PoissoneqD}
	\begin{equation*}
		\begin{cases}
		    \frlap u= h   \qquad &\mbox{ in } {B_r} ,
 		\\  u= 0   \qquad &\mbox{ in } {\obal}.
	\end{cases}
	\end{equation*}
\end{thm*}	
The proof is classical, and makes use of the properties and representation formulas involving the two functions $\Phi$ and $P_r$.

We are also interested in the values of the normalization constants that appear in the definitions of the $s$-mean kernel (and the Poisson kernel) and of the fundamental solution. {We will deal separately with the two cases $n\neq2s$ and $n=2s$. We have the following definition}:

\begin{definition}
\label{definition:ctn} The constant $a(n,s)$ introduced in definition \eqref{fundsolution} is
\begin{align}
        a(n,s) : &= \displaystyle { \frac{ \Gamma (\frac{n}{2}-s)} { 2^{2s}\pi ^ {\frac{n}{2} }  \Gamma(s)}} &  \text{ for } &n\neq 2s \label{ctans1},\\
	a\Big(1,\frac{1}{2}\Big) :&= { -\frac{1}{\pi}} & \text{ for } &n=2s. \label{ctans3}
	\end{align}

	The constant $c(n,s)$ introduced in definition \eqref{smeandefn} is
\begin{equation}
	c(n,s) :=\frac {\Gamma (\frac{n}{2}) \sin \pi s }  {\pi^{\frac{n}{2}+1} }.\label{ctcns}
\end{equation}

\end{definition}

These constants are used for normalization purposes, and we explicitly clarify how their values arise. However, these values are only needed to compute the constant $\kappa(n,s)$ introduced in Theorem \ref{theorem:thm1}, and have no role for the rest of our discussion. The value of $\kappa(n,s)$ is given in terms of the Euler-Gamma function as:

\begin{thm*}[\ref{theorem:kns}] The constant $\kappa(n,s)$ introduced in identity \eqref{forgkns} is
	\begin{equation*} \begin{aligned}
	 \kappa(n,s) &= \displaystyle {\frac{\Gamma(\frac{n}{2}) } {2^{2s}\pi^{ \frac{n}2}   \Gamma^2(s) } }& \text{ for } &n \neq 2s, \\
		\kappa\Big(1,\frac{1}{2}\Big) &={\frac{1}{\pi}}
		  & \text{ for } &n=2s.
		  \end{aligned}
		  \end{equation*}
\end{thm*}

One interesting thing that we want to point out here is related to the two constants $C(n,s)$ and $c(n,s)$. The constant $C(n,s)$ is given in \cite{galattica} in the definition of the fractional Laplacian, and we defined it here in \eqref{GCNS}. The constant $c(n,s)$ is introduced in \cite{Landkof} in the definition of the $s$-mean kernel and the Poisson-kernel, and is here given in \eqref{ctcns}. These two constants are used for different normalization purposes. We observe that they have similar asymptotic properties, however they are not equal. {In the following proposition we explicitly compute the value of $C(n,s)$ defined in \eqref{GCNS}.}

\begin{thm*}[\ref{thm:Cc}]
The constant $C(n,s)$ introduced in \eqref{GCNS} is given by
	\begin{equation} \label{cnscomputed} C(n,s) = {\frac{2^{2s} s \Gamma\left(\frac{n}2 +s\right)} {\pi^{\frac{n}2} \Gamma(1-s)}}.\end{equation}

\end{thm*}

This paper is structured as follows: in the Preliminaries (Section \ref{preliminaries}) we introduce some kernels related to the fractional Laplacian. In Subsection \ref{smean} we define the $s$-mean value property by means of the $s$-mean kernel and prove that if a function is $s$-harmonic, then it has the $s$-mean value property. Subsection \ref{fundsol} deals with the study of the function $\Phi$ as the fundamental solution of the fractional Laplacian.
The fractional Poisson kernel is introduced in Subsection \ref{thepoissonkernel}, and the representation formula for equation \eqref{LaplaceeqD} is obtained. Section \ref{green} focuses on the Green function, presenting two main theorems: Theorem \ref{theorem:thm1} gives a more basic formula of the function $G$ for the ball and is treated in Subsection \ref{thm1}; Theorem \ref{theorem:thm2}, that illustrates how the solution to equation \eqref{PoissoneqD} is built by means of the function $G$, is dealt with in Subsection \ref{thm2}.  The computation of the normalization constants introduced along this notes is done in Subsection \ref{constants}.
The Appendix \ref{appendix} introduces three special functions (Gamma, Beta and hypergeometric), the point inversion transformations and some calculus identities that we use throughout this paper.

\section[Preliminaries]{Preliminaries} \label{preliminaries}
In this section, we deal with the $s$-mean kernel, the fundamental solution and the Poisson kernel.

Let $s\in (0,1)$.

\subsection{The $s$-mean value property}\label{smean}

In this subsection, we define the $s$-mean value property of the function $u$, namely, an average property defined by convolution of $u$ with the $s$-mean kernel.

\begin{definition}[$s$-mean value property]
Let $x\in \Rn$. We say that $u $ belonging to $ L_s^1(\Rn)$ and continuous in a neighborhood of $x$ has the $s$-mean value property at $ x$ if, for any $r>0$ arbitrarily small,
	\begin{equation}  u(x)= A_r* u (x) .\label{smvp}\end{equation}
We say that $u$ has the $s$-mean value property in $\Omega \subseteq {\Rn}$ if for any $r>0$ arbitrarily small, identity \eqref{smvp} is satisfied at any point $x \in \Omega$.
\end{definition}

The kernel $A_r$ defined in \eqref{smeandefn} is used to state the $s$-mean value property, which makes it reasonable to say that $A_r$ plays the role of the $s$-mean kernel. The main result that we state here is that if a function has the $s$-mean value property, then it is $s$-harmonic  (i.e $u$ satisfies the classical relation $\frlap u =0$).

\begin{theorem}
\label{theorem:arm}
Let $u \in L_s^1(\Rn)$ be $C^{2s+\eee}$ in a neighborhood of $x\in \Rn$.
If $u$ has the $s$-mean value property at $x$, then $u$ is $s$-harmonic at $x$.
\end{theorem}

\begin{proof}[Proof]
The function $u$ has the $s$-mean value property for any $r>0$ arbitrarily small, namely
	 \begin{equation*}
           u(x)=  A_r* u (x) =  \int_{  \obal} A_r(y)u(x-y) \, dy.
	\end{equation*}
Using identity \eqref{Ir} we obtain that
	\begin{equation*}
		\begin{split}
	0= u(x)  - \int_{\obal}   A_r(y)  u(x-y) \,  dy  = c(n,s) r^{2s}  \int_{\obal}   \frac{u(x) -u(x-y)}{{(|y|^2-r^2)}^s|y|^n}  \, dy,
		\end{split}
	\end{equation*}
thus, since $r>0$
 	\begin{equation}  \int_{\obal}  \frac {u(x)-u(x-y)  }{(|y|^2-r^2)^s |y|^n}\,  dy =0.\label{zeroeq} \end{equation}
 Hence, in order to obtain $\frlap u(x)=0$ we prove that
 	\begin{equation} \lim_{r \to 0} \int_{\Rn\setminus B_r} \frac{u(x)-u(x-y)}{|y|^{n+2s}} \, dy = \lim_{r \to 0} \int_{\obal}  \frac {u(x)-u(x-y)  }{(|y|^2-r^2)^s |y|^n}\,  dy \label{clsmfrl}. \end{equation}
Let $R>r\sqrt{2}$. We write the integral in \eqref{zeroeq} as
	\begin{equation}\label{eq1112}
		\begin{split}
		 \int_{\obal}   &\frac{u(x)-u(x-y)}{(|y|^2-r^2)^s |y|^n}\, dy \\ = \; &\int_{{\Rn}\setminus B_{R}}  \frac { u(x)-u(x-y) }{(|y|^2-r^2)^s |y|^n}\, dy
	 + \int_{B_{R}\setminus B_r} \frac { u(x)-u(x-y) }{(|y|^2-r^2)^s |y|^n}\, dy \\
	=\;  &I_1(r,R)+I_2(r,R).
		\end{split}
	\end{equation}
In $I_1(r,R)$ we see that $ \frac{|y|^2 }{|y|^2-r^2} <2$ and obtain that
	\[ \frac{|u(x)-u(x-y)|} {(|y|^2-r^2)^s |y|^n}   \leq 2^{s}\frac{ |u(x)-u(x-y) |} {|y|^{n+2s}}  \in L^1(\Rn \setminus B_R, \,dy),\]
	as  $u\in L_s^1(\Rn)$. We can use the dominated convergence theorem, send $r \to 0$  and conclude that \begin{equation}\label{eq1111} \lim_{r\to 0} I_1(r,R)=\int_{\Rn\setminus B_{R}} \frac{u(x)-u(x-y)}{|y|^{n+2s}} \, dy .\end{equation}
Now, for $r<|y|<R$ and $u\in C^{2s+\eee}$ (for $s<1/2$) in a neighborhood of $x$  we have the bound
	\[ \begin{split}
		\big| u(x)-u(x-y) \big|\leq c |y|^{2s+\eee},  \end{split}\]
while for $s\geq1/2$ and $u \in C^{1,2s+\eee-1}$ we use that
	\[\begin{split}
		 |u(x) -u(x-y) - y \cdot \nabla u(x)|\; =&\; \Big| \int_0^1 y \big(\nabla u(x-ty) -\nabla u(x) \big) \, dt\Big|\\
		 	 \leq &\; |y| \int_0^1 \Big| \nabla u(x-ty) -\nabla u(x) \Big| \,  \, dt \\
		 	 \leq &\;c(s,\eee) |y|^{2s+\eee}. \end{split} \]
Notice that $ \frac{ y \cdot \nabla u(x)}{(|y|^2-r^2)^s |y|^n} $ and $\frac{ y \cdot \nabla u(x)}{|y|^{2s+n}} $ are even functions, hence they vanish when integrated on the symmetrical domain $B_{R} \setminus B_r$. Therefore, by setting
	\begin{equation}\label{eq1113}  J(r,R):=I_2(r,R) - \int_{B_R \setminus B_r} \frac{u(x)-u(x-y) } {|y|^{2s+n}}  \,dy\end{equation}
we have that
	\[ \begin{split}    J(r,R) = \int_{B_R \setminus B_r} \Bigg( \frac{u(x)\!-\!u(x\!-\!y)\!-\!y \cdot \nabla u(x)}{(|y|^2\!-\!r^2)^s |y|^n}  -\frac{u(x)\!-\!u(x\!-\!y) \!-\!y \cdot \nabla u(x)} {|y|^{2s+n}} \Bigg)\,dy
	\end{split}\]
and by passing to polar coordinates and afterwards making the change of variables $\rho =rt$ we get
\[ \begin{split}
		|J(r,R)| \leq &\; c (s,\eee) \int_{B_R \setminus B_r} |y|^{2s+\eee}  \Big((|y|^2-r^2)^{-s} |y|^{-n} - |y|^{-n-2s}\Big)\, dy \\
				= &\; {c}(n,s,\eee) \int_r^R \rho ^{\eee-1} \bigg( \frac{\rho^{2s}}{(\rho^2-r^2)^{s}} -1\bigg) \, d\rho \\
				< 	&\; {c}(n,s,\eee) r^{\eee} \int_1^{\frac{R}{r}} t^{\eee-1}  \bigg( \frac{t^{s}}{(t-1)^s} -1\bigg) \, dt	
\end{split}\]
since $t/(t+1)>1$. Now for $ t \in (1, \sqrt 2)$ we have that
	\[ \begin{split}   \int_1^{\sqrt 2} t^{\eee-1} \bigg( \frac{t^s}{(t-1)^s} -1\bigg) \, dt  \leq c(s) \int_1^{\sqrt 2} \bigg( (t-1)^{-s} -t^{-s}\bigg) \, dt = \tilde c(s).
	\end{split} \]
On the other hand, for $t\geq\sqrt 2$
\[\Big(1- \frac{1}{t}\Big)^{-s} - 1 \leq \frac{s}{t} \Big(1-\frac{1}{\sqrt 2} \Big)^{-s-1} \] and we have that
	 \[ \begin{split}   \lim_{r \to 0} \int_{\sqrt 2}^{\frac{R}{r}} t^{\eee-1} \bigg( \frac{t^s}{(t-1)^s} -1\bigg) \, dt   \leq &\;  \int_{\sqrt 2}^{\infty}  t^{\eee-1}  \Bigg( \Big( 1- \frac{1}{t}\Big)^{-s} -1 \Bigg) \, dt\\
\leq &\;c(s) \int_{\sqrt 2} ^{\infty} t^{\eee -2 }\, dt = \bar c(s,\eee).
	\end{split} \]
Thus by sending $r \to 0 $ we obtain that
	\[ \begin{split}   \lim_{ r \to 0} J(r,R) =0
	\end{split} \]
and therefore in \eqref{eq1113} \[ \lim_{r \to 0}I_2(r,R) =  \lim_{r \to 0} \int_{B_R \setminus B_r} \frac{u(x)-u(x-y) } {|y|^{2s+n}}  \,dy.\]
Using this and \eqref{eq1111} and passing to the limit in \eqref{eq1112}, claim \eqref{clsmfrl} follows and hence  the conclusion that $\frlap u(x)=0$.
\end{proof}

\subsection{The fundamental solution}\label{fundsol}

We claim that the function $\Phi$ plays the role of the fundamental solution of the fractional Laplacian, namely the fractional Laplacian of $\Phi$ is equal in the distributional sense to the Dirac Delta function evaluated at zero. The following theorem provides the motivation for this claim.

\begin{theorem}
\label{theorem:thm3} In the distributional sense (given by definition \eqref{disf1})
	 \[\frlap \Phi =\delta_0 .\]
\end{theorem}

The computation of the Fourier transform of the fundamental solution is required in order to prove Theorem \ref{theorem:thm3}.
\begin{proposition}
\label{proposition:ansss}
a) For $n>2s$, let $f \in L^1(\Rn)\cap C(\Rn)$ with $\wck f \in \Sa_s(\Rn)$,	\\
b) for $n\leq 2s$, let $f\in L^1(\R)\cap C(\R) \cap C^1\big( (-\infty, 0) \cup (0,+\infty)\big)$ with $\wck f\in \Sa_s(\R)$ such that
	 \begin{equation}\label{condffff}\begin{aligned}
	 &|f(x)|\leq c_1|x|^{2s} \quad &\text{ for } & x \in \R\\
	   & |f(x)|\leq \frac{c_2}{|x|} \quad &\text{ for } &|x|>1 \\
	  	 &|f'(x)|\leq c'_1|x|^{2s-1} \quad &\text{ for } &0 <|x|\leq 1\\
	  &  |f'(x)|\leq \frac{c'_2}{|x|} \quad &\text{ for } &|x|>1. \end{aligned} \end{equation}
 Then
\[\int_{\Rn} \Phi(x) \wck f(x) \, dx =\int_{\Rn} ({2\pi} |x|)^{-2s} f(x) \, dx.\]
\end{proposition}

\begin{proof}
We notice that the conditions \eqref{condffff} on $f$ assure that the integrals are well defined. Indeed, since $\Phi \in L_s^1(\Rn) \subset \Sa'_s(\Rn)$ the left hand side is finite thanks to \eqref{db1}. The right hand side is also finite since, for $n>2s$,
	\[ \begin{split} \int_{\Rn} |f(x)||x|^{-2s} \, dx  \leq &\; c_n\sup_{x\in B_1} |f(x)| \int_0^1 \rho^{n-2s-1}\, d\rho + \int_{\Rn \setminus B_1}  |f(x)| |x|^{-2s}\, dx\\
		\leq &\; c_n\sup_{x \in B_1} |f(x)| + \|f\|_{L^1(\Rn)} \end{split}\]
		and for $n\leq 2s$  we have that
	\[ \begin{split}\int_{\R} |f(x)||x|^{-2s} \, dx  \leq &\; \int_{\R\setminus B_1} |f(x)| |x|^{-2s} \, dx+ c_1 \int_{B_1}  dx \\
							\leq &\; \|f\|_{L^1(\R)} + 2c_1.\end{split}\]

a) For $n>2s$ we prove that
	\begin{equation} \label{firstcl111} a(n,s)\int_{\Rn} |x|^{-n+2s} \wck {f}(x)\, dx=\int_{\Rn} ({2\pi}|x|)^{-2s} f(x)\, dx.\end{equation}
We use the Fourier transform of the Gaussian distribution as the starting point of the proof. For any $ \delta > 0$ we have that
	\[    \mathcal{F} (e^{-\pi \delta |x|^2})  = \delta^{- \frac{n}{2}} e^{-\pi\frac{|x|^2}{\delta}}.\]
In particular for any $ f \in L^1(\Rn)$ and $\wck {f}\in \Sa_s(\Rn)$ (which is a subspace of $L^2(\Rn)$), by Parseval identity we obtain
	\[ \int_{\Rn}  e^{-\pi \delta |x|^2} \wck {f}(x) \, dx= \int_{\Rn} \delta^{-\frac{n}{2}} e^{-\pi\frac{|x|^2}{\delta}}  f(x)\, dx. \]
We multiply both sides by $\delta^{\frac{n}{2} -s-1}$, integrate in $\delta$ from $0$ to $\infty$. We use the notations
	\[I_1=\int_0^{\infty} \Bigg( \int_{\Rn} \delta^{\frac{n}{2} - s -1} e^{-\pi\delta |x|^2}\wck {f} (x)  \, dx\Bigg) \, d\delta \]
and
	\[I_2= \int_0^{\infty} \Bigg( \int_{\Rn} \delta^{- s -1} e^{-\pi \frac {|x|^2}{\delta}}   f(x) \, dx\Bigg) \, d\delta,\] having $I_1=I_2$.
We perform in $I_1$ the change of variable $\alpha =\delta|x|^2$ and obtain that
	\[I_1 =  \int_{\Rn}|x|^{-n+2s} \wck{f} (x)  \Bigg(\int_0^{\infty} \alpha^{\frac{n}{2} -s-1}  e^{-\pi\alpha}\, d\alpha \Bigg)\, dx. \]
We set	\begin{equation} \label{c1} c_1 := \int_0^{\infty} \alpha^{\frac{n}{2} -s-1}  e^{-\pi\alpha}\, d\alpha, \end{equation} which is a finite quantity since $ \frac{n}{2} -s-1 > -1$.
On the other hand in $I_2$ we change the variable $ \displaystyle \alpha = {|x|^2}/{\delta} $ and obtain that
	\[I_2 =  \int_{\Rn}   |x|^{-2s}  f(x) \Bigg(\int_0^{\infty} \alpha^{s-1}  e^{-\pi\alpha}\, d\alpha \Bigg)\, dx. \]
We then set
	\begin{equation} c_2 :=  \int_0^{\infty} \alpha^{s-1}  e^{-\pi\alpha}\, d\alpha\label{c2},  \end{equation} which is finite since $ s-1 > -1$.
As $I_1=I_2$ it yields that
	\begin{equation*}  \frac{c_1}{c_2 ({2\pi})^{2s}} \int_{\Rn} |x|^{-n+2s}  \wck{f} (x)\, dx =  \int_{\Rn} ({2\pi}|x|)^{-2s}   f(x)\, dx . \end{equation*}
We take
	\begin{equation}\label{aaargh1} a(n,s) =  \frac{c_1}{c_2 ({2\pi})^{2s}}\end{equation}and the claim \eqref{firstcl111} follows. This concludes the proof for $n>2s$.

b) For $n<2s$ (hence $n=1$ and $s>1/2$), let $R>0$ be as large as we wish (we will make $R$ go to $\infty$ in sequel). Then
	\[ \begin{split}
		\int_{B_R} |x|^{2s-1} \wck f (x) \, dx  = &\;\int_0^R x^{2s-1} \Big( \wck f(x) + \wck f(-x)\Big) \, dx \\
		= &\;2\int_0^R x^{2s-1} \int_{\R} f(\xi) \cos(2\pi  \xi x) \, d\xi \, dx \\
		=&\; 2 \int_{\R} f(\xi) \left( \int_0^R x^{2s-1} \cos(2\pi \xi x) \, dx\right) \, d\xi.\end{split} \]
{We use the change of variables $\bar x = 2 \pi x$ (but still write $x$ as the variable of integration for simplicity) and let $\bar R = 2 \pi R$}. Then
 \[\int_0^R x^{2s-1} \cos({2\pi} \xi x) \, dx = {(2\pi)^{-2s}} \int_0^{\bar R} x^{2s-1} \cos(\xi x) \, dx.\]
 We have that
 	\[ 	\int_{B_R} |x|^{2s-1} \wck f (x) \, dx  ={\frac{2^{1-2s}}{\pi^{2s}} }\int_{\R} f(\xi) \left( \int_0^{\bar R} x^{2s-1} \cos(\xi x)\, dx  \right)\, d\xi.\]
Integrating by parts and changing variables $|\xi| x=t$ we obtain that
	\[ \begin{split}
				\int_0^{\bar R} x^{2s-1} \cos(\xi x) \, dx = &\;x^{2s-1} \frac{\sin(\xi x)}{\xi} \bigg|_0^{\bar R} -(2s-1) \int_0^{\bar R} x^{2s-2} \frac{\sin(\xi x)}{\xi} \, dx \\
				=&\; {\bar R}^{2s-1} \frac{\sin(\xi {\bar R})}{\xi} -(2s-1)|\xi|^{-2s} \int_0^{{\bar R}|\xi|} t^{2s-2} \sin t \, dt.
	\end{split} \]
	Therefore
	\begin{equation} \begin{split}\label{wckn1}
		\int_{B_R} |x|^{2s-1} \wck f(x)\, dx = &\;\frac{2^{1-2s}}{\pi^{2s} } {\bar R}^{2s-1} \int_{\R} f(\xi)\frac{\sin(\xi {\bar R})}{\xi} \, d\xi \\
			&\;- \frac{2^{1-2s}(2s-1) }{\pi^{2s} }  \int_{\R} f(\xi) |\xi |^{-2s} \bigg( \int_0^{{\bar R}|\xi|} t^{2s-2} \sin t \, dt \bigg)\, d\xi. \end{split}\end{equation}
We claim that
	\begin{equation}\label{n1c1} \lim_{R \to \infty} {\bar R}^{2s-1}\int_{\R} f(\xi)\frac{\sin(\xi {\bar R})}{\xi} \, d\xi=0 . \end{equation}
	We integrate by parts and obtain that
	\[ \begin{split}
		\bigg|\int_0^{\infty} f(\xi) \frac{\sin(\xi {\bar R})}{\xi} \, d\xi \bigg| \leq &\;  \frac{|f(\xi)|}{\xi} \frac{|\cos(\xi {\bar R})|}{{\bar R}} \bigg|_{0}^{\infty} + \frac{1}{{\bar R}}   \bigg( \int_{0}^{\infty}  |\cos(\xi {\bar R})| \frac{|f(\xi)|}{\xi^2} \, d\xi   \\
		&\; +\int_{0}^{\infty} |\cos(\xi {\bar R})| \frac{|f'(\xi)|}{\xi} \, d\xi\bigg).
	\end{split}\]
By \eqref{condffff}, for $\xi $ large we have that
	 \[ \frac{|f(\xi)|}{\xi}|\cos(\xi {\bar R})|  \leq c_2\frac{|\cos(\xi {\bar R})| }{\xi^2}, \quad \mbox{hence} \quad \lim_{\xi \to  \infty}  \frac{|f(\xi)|}{\xi} \frac{|\cos(\xi {\bar R})|}{{\bar R}} =0.\]
	  For $\xi $ small we have that \[ \frac{|f(\xi)|}{\xi} \leq c _1\xi^{2s-1},\quad \mbox{hence} \quad  \lim_{\xi \to 0} \frac{|f(\xi)|}{\xi} \frac{|\cos(\xi {\bar R})|}{{\bar R}} =0.\]
Furthermore, by changing variables $t=\xi {\bar R}$ (and noticing that the constants may change value from line to line) we have that
\[\begin{split}
	\int_0^{\infty} \frac{|f(\xi)|}{\xi^2} |\cos (\xi {\bar R}) |\, d\xi \leq &\;c_1 \int_0^1   \xi^{2s-2}  |\cos (\xi {\bar R}) | \, d\xi  +c_2  \int_1^{\infty}  \xi^{-3}  |\cos (\xi {\bar R}) |\, d\xi \\
	\leq &\;c_1 {\bar R}^{1-2s} \int_0^{\bar R} t^{2s-2}  |\cos t |\, dt +c_2 {\bar R}^2  \int_{\bar R}^{\infty} t^{-3}|\cos t| dt
	\leq   \frac{c}{2}
	\end{split}\]
and \[\begin{split}
	\int_0^{\infty} \frac{|f'(\xi)|}{\xi} |\cos (\xi {\bar R}) |\, d\xi \leq &\;c_1' \int_0^1 	  \xi^{2s-2} |\cos (\xi {\bar R}) |  \, d\xi  +c_2' \int_1^{\infty}  \xi^{-2} |\cos (\xi {\bar R}) |  \, d\xi \\
	\leq &\; c_1' {\bar R}^{1-2s} \int_0^{\bar R} t^{2s-2}|\cos t| \, dt +c_2'{\bar R} \int_{\bar R}^{\infty}t^{-2} |\cos t|\, dt
	 \leq  \frac{c}{2}.
	\end{split}\]
Hence \[\bigg| \int_0^\infty  f(\xi) \frac{\sin(\xi{\bar R})}{\xi} \, d\xi \bigg| \leq \frac{c}{R},\]
and in the same way we obtain
	\[\bigg|\int_{-\infty}^0  f(\xi) \frac{\sin(\xi {\bar R})}{\xi} \, d\xi  \bigg|= \bigg|\int_0^\infty  f(-\xi) \frac{\sin(\xi {\bar R})}{\xi} \, d\xi  \bigg| \leq \frac{c}{R}.\] Therefore
	\[ \lim_{R\to \infty} {\bar R}^{2s-1} \int_{\R}  f(\xi) \frac{\sin(\xi {\bar R})}{\xi} \, d\xi   = 0 \] and we have proved the claim \eqref{n1c1}.
	Now we claim that (and this holds also for $n= 2s$)
\begin{equation}\label{n1c2}   \lim_{R \to \infty} \int_\R f(\xi)|\xi|^{-2s}  \Big(\int_0^{{\bar R}|\xi|} t^{2s-2} \sin t \, dt\Big) \, d\xi = - \cos (\pi s) \Gamma(2s-1) \int_\R f(\xi)|\xi|^{-2s}   \, d\xi .\end{equation} In order to prove this, we estimate the difference
	\[ \begin{split}
		\bigg| \int_0^\infty t^{2s-2} \sin t \, dt &\; -\int_0^{{\bar R}|\xi|} t^{2s-2} \sin t \, dt \bigg| \leq  \bigg| \int_{{\bar R}|\xi|}^\infty t^{2s-2} \sin t \, dt\bigg|\\
		\leq &\; |t ^{2s-2} \cos t |\bigg|_{{\bar R}|\xi|}^\infty + (2s-2) \int_{{\bar R}|\xi|}^\infty |t|^{2s-3} | \cos t | \, dt
		\leq  c({\bar R}|\xi|)^{2s-2}.
	\end{split}\]
We then have that
	\[ \begin{split}
	\bigg|	\int_\R f(\xi)|\xi|^{-2s}  &\Big( \int_0^\infty t^{2s-2} \sin t \, dt  -\int_0^{{\bar R}|\xi|} t^{2s-2} \sin t \, dt \Big)  \, d\xi \bigg| \\
	\leq&\; c {\bar R}^{2s-2} \int_\R |f(\xi)| |\xi|^{-2}\, d\xi \\
	\leq &\; c {\bar R}^{2s-2}\bigg( c_1\int_0^1  \xi^{2s-2} \, d\xi + c_2\int_1^\infty \xi^{-3}\, d\xi \bigg)
	= R^{2s-2} \overline c. \end{split}\]
	Hence we obtain
		\[  \lim_{R \to \infty} \int_\R f(\xi)|\xi|^{-2s}  \Big(\int_0^{{\bar R}|\xi|}t^{2s-2} \sin t \, dt\Big) \, d\xi = \int_\R f(\xi)|\xi|^{-2s} \left(\int_0^{\infty}t^{2s-2} \sin t \, dt \right)  \, d\xi \]
{and the claim \eqref{n1c2} follows from the identity \eqref{ctcomp1111} in the Appendix. }

By sending $R$ to infinity in \eqref{wckn1} we finally obtain  that
		\begin{equation} \begin{split} \label{oba1s}
			\int_{\R} |x|^{2s-1} \wck f(x)\, dx =&\; \frac{2^{1-2s}}{\pi^{2s}} (2s-1) \cos(\pi s) \Gamma(2s-1) \int_\R |\xi|^{-2s} f(\xi)\, d\xi \\
				=&\;   2\cos(\pi s) \Gamma(2s) \int_\R (2\pi|\xi|)^{-2s} f(\xi)\, d\xi.\end{split}\end{equation}
Therefore {taking $ a(1,s)  =({2\cos(\pi s)\Gamma(2s)})^{-1}$ we obtain that}
	\[ a(1,s) \int_{\R} |x|^{2s-1} \wck f(x)\, dx = \int_\R ({2\pi} |x|)^{-2s} f(x)\, dx ,\]
			hence the result for $n<2s$.
			
Now, for $n=2s$ we have that
			\[ \int_{B_R} \log |x| \wck f(x)\, dx = 2\int_{\R} f(\xi) \left(\int_0^R \log x \cos({2 \pi} \xi x) \, dx \right)\, d\xi .\]
		{We change variables $\bar x = 2 \pi x$ (but still write $x$ as the variable of integration for simplicity) and let $\bar R = 2 \pi R$}. Then we have that
		\[ \begin{split} \int_0^R \log x\cos(2\pi \xi x)\, dx = &\; \int_0^{\bar R } \Big(\log x -\log (2\pi) \Big) \cos(\xi x)\, \frac{dx}{2\pi}\\
			= &\; \frac{1}{2\pi}  \int_0^{\bar R}  \log x\cos(\xi x) \, dx - \frac{\log(2\pi)}{2\pi} \frac{\sin(\xi \bar R)}{\xi}.
			\end{split}\]
	We integrate by parts and obtain that
			\[\int_0^{\bar R} \log x\cos(\xi x) \, dx = \log {\bar R} \frac{ \sin (\xi {\bar R})}{\xi} -\frac{1}{|\xi|} \int_0^{{\bar R}|\xi|} \frac{\sin t }{t}\,dt.\]
	We thus have that
			\[ \int_{B_R} \log |x| \wck f(x)\, dx = \frac{1}{\pi} \log R \int_{\R} f(\xi) \frac{\sin(\xi \bar R)}{\xi} \, d\xi - \frac{1}{\pi} \int_{\R} f(\xi) |\xi|^{-1} \left (\int_0^{\bar R |\xi|} \frac{\sin t}{t} \, dt \right)\, d\xi.\]
			We claim that\begin{equation} \label{n2c1} \lim_{R\to \infty} \log R \int_{\R} f(\xi) \frac{ \sin (\xi {\bar R} )}{\xi} \, d\xi =0  .\end{equation}
	Indeed we have that \[ \bigg|\int_0^{1/{\bar R}} f(\xi) \frac{\sin (\xi {\bar R})}{\xi} \, d\xi \bigg| \leq  \int_0^{1/{\bar R}} |f(\xi)| \frac{\xi {\bar R}}{\xi} \, d\xi \leq c_1 {\bar R}\int_0^{1/{\bar R}} \xi \, d\xi = \frac{  c_1}{\bar R}.\]
	  Moreover  integrating by parts we have that
	 \[ \begin{split}
	 	\bigg| \int_{1/{\bar R}}^{\infty} f(\xi) \frac{\sin (\xi {\bar R})}{\xi} \, d\xi  \bigg| \leq &\;  \frac{|f(\xi)|}{\xi } \frac{| \cos(\xi {\bar R})| } {{\bar R}} \Bigg|_{1/{\bar R}}^{\infty}  +\frac{1}{{\bar R}} \Bigg( \int_{1/{\bar R}}^{\infty} \frac{|f(\xi)|}{\xi^2} |\cos(\xi {\bar R})|\, d\xi\\
	 	&\; + \int_{1/{\bar R}}^{\infty} \frac{|f'(\xi)|}{\xi}| \cos(\xi {\bar R})| \, d\xi \Bigg).
		\end{split}\]
	We have that for $\xi$ large
			 	\[  \frac{|f(\xi)|}{\xi } \frac{|\cos(\xi {\bar R})|}{{\bar R}} \leq \frac{c_2}{\bar R} \frac{|\cos(\xi {\bar R})|}{\xi^2} \] hence
			 	\[ \lim_{\xi\to \infty} \frac{|f(\xi)|}{\xi } \frac{|\cos(\xi {\bar R})|} {{\bar R}}  =0.\]
	On the other hand by using the change of variables $t=\xi {\bar R}$
		\[ \begin{split} \int_{1/{\bar R}}^{\infty} \frac{|f(\xi)|}{\xi^2}|\cos(\xi {\bar R})| \, d\xi \leq &\;c_1 \int_{1/{\bar R}}^{1} \frac{ |\cos(\xi {\bar R})| }{\xi}\, d\xi +c_2 \int_1^{\infty} \frac{|\cos(\xi {\bar R})|}{\xi^3} \, d\xi\\
		  =&\; c_1 \int_1^{\bar R} \frac{|\cos t|}{t} \, dt + c_2 \int_{\bar R}^{\infty} \frac{{\bar R}^2}{t^3} |\cos t|\, dt
		  \leq \bar c_1 \log R +\bar c_2.
		\end{split}\]
			 Moreover
			 \[ \begin{split} \int_{1/{\bar R}}^{\infty} \frac{|f'(\xi)|}{\xi}|\cos(\xi {\bar R})| \, d\xi \leq &\;c'_1 \int_{1/{\bar R}}^{1} \frac{|\cos(\xi {\bar R})|}{\xi} \, d\xi +c'_2 \int_1^{\infty} \frac{|\cos(\xi {\bar R})| }{\xi^2} \, d\xi\\
		  =&\; c'_1 \int_1^{\bar R} \frac{|\cos t|}{t} \, dt + c'_2 \int_{\bar R}^{\infty} \frac{{\bar R}}{t^2} |\cos t|\, dt
		  \leq\bar c'_1 \log { R} +\bar c'_2.
		\end{split}\]
		Hence \[ \lim_{R \to \infty} \log R \int_0^{\infty} f(\xi) \frac{\sin (\xi {\bar R})}{\xi} \, d\xi =0\] and since the same bounds hold for $ \int_{-\infty}^0  f(\xi) \frac{\sin (\xi {\bar R})}{\xi} \, d\xi ,$  the claim \eqref{n2c1} follows.
	Also, the proof of the claim \eqref{n1c2} assures us that  \[ \begin{split} \lim_{R \to \infty} \int_0^{\infty} f(\xi) |\xi|^{-1}  \bigg(\int_0^{{\bar R}|\xi|} \frac{\sin t }{t}\,dt \bigg) \, d\xi  = &\; \int_0^{\infty} f(\xi) |\xi|^{-1} \bigg(\int_0^{\infty} \frac{\sin t }{t}\,dt \bigg) \, d\xi   \\
			  =&\; \frac{\pi}{2} \int_0^{\infty} f(\xi) |\xi|^{-1} \, d \xi .\end{split}\] It follows that
			\begin{equation*} \int_{\R} \log |x| \wck f(x)\, dx = -\frac{1}{2}  \int_\R |\xi|^{-1} f(\xi)\, d\xi,  \end{equation*} 	
				hence 				
				\begin{equation} \label{obans3} -\frac{1}\pi  \int_{\R} \log |x| \wck f(x)\, dx = \int_\R ({2\pi} |\xi|)^{-1} f(\xi)\, d\xi  \end{equation}
				and the result holds for $n=2s$.				This concludes the proof of the Proposition.
	\end{proof}

\begin{remark}
It is now clear that we have chosen $a(n,s)$ in Definition \ref{definition:ctn} in order to normalize the Fourier transform of the fundamental solution. Indeed, for $n>2s$, we perform the change of variable $\pi \alpha =t$ in \eqref{c1} and by \eqref{gamma} we obtain that
	\[ c_1= \pi^{ s-\frac{n}{2}} \int_0^{\infty} t^{\frac{n}{2}-s-1} e^{-t} \, dt = \pi^{ s-\frac{n}{2}}  \Gamma \bigg( \frac{n}{2} -s\bigg).\]
Also in \eqref{c2} we change the variable $\pi \alpha =t$ and get that
	\[c_2= \pi^{ -s}  \int_0^{\infty}t^{s-1} e^{-t}\, dt= \pi^{ -s} \Gamma(s).\]
Therefore
	 \begin{equation*} \frac{c_1}{c_2}= \frac{ \pi ^ {2s-\frac{n}{2} }  \Gamma (\frac{n}{2}-s)} {\Gamma(s)}, \quad  \mbox{hence by \eqref{aaargh1}}\quad   a(n,s)= {\frac{ \Gamma (\frac{n}{2}-s)} {2^{2s}\pi^{\frac{n}2} \Gamma(s)}}.\end{equation*}
The value $a(1,s)$ is computed in \eqref{oba1s}. We point out that we can rewrite this value using \eqref{gam1} and \eqref{gam2}, as follows
	\[\begin{aligned}  a(1,s) = \frac{1}{2\cos(\pi s)\Gamma(2s) }= \frac{\Gamma(1/2-s) }{2^{2s} \sqrt{\pi} \Gamma(s)} .\end{aligned}\]
	Moreover, we observe that identity \eqref{obans3} says that
				\[ a\left(1,\frac{1}{2}\right) = -\frac{1}{\pi}.  \]
		\end{remark}		

By applying this latter Proposition \ref{proposition:ansss}, we prove Theorem \ref{theorem:thm3}.
\begin{proof}[Proof of Theorem \ref{theorem:thm3}]
For any $f\in \mathcal{S}(\Rn)$ we have that $\F^{-1} \Big(|\xi|^{2s} \widehat f(\xi) \Big) \in \Sa_s(\Rn) $ (according to definition \eqref{frlaphdef} and to \eqref{frb1}). Notice that $|\xi|^{2s} \widehat f(\xi) \in L^1(\Rn)\cap C(\Rn)$, since
	\[ \int_{\Rn} |x|^{2s}|\widehat f(x)|\, dx\leq [\widehat f]_{\Sa(\Rn)}^{0,n+2} \int_{\Rn \setminus B_1} |x|^{2s-n-2} \, dx + \sup_{x\in B_1} |\widehat f(x)| \leq c(f), \]where we use the seminorm defined in \eqref{seminormss}.
Moreover, for $n\leq 2s$ we have that
	 	\[\begin{aligned}
	 		&  |\xi|^{2s} |\widehat f(\xi)| \leq \|\widehat f\|_{L^{\infty} (\R)} |\xi|^{2s} = c_1 |\xi|^{2s} \quad &\text{ for } &\xi\in \R,\\
	 		&  |\xi|^{2s} |\widehat f(\xi)| \leq [\widehat f(\xi)]_{\Sa(\R)}^{0,3} |\xi|^{2s-3} \leq \frac{c_2}{|\xi|} \quad &\text{ for } &|\xi| > 1.\end{aligned}\]  Also for $0 \neq|\xi|\leq 1$
	 		\[\begin{split} \Big| \frac{d}{d\xi} \big( |\xi|^{2s}\widehat f(\xi)\big) \Big| \leq &\; 2s |\xi|^{2s-1} |\widehat f(\xi)| + |\xi|^{2s}\Big|\frac{d}{d\xi} \widehat f (\xi)\Big| \\
	 		\leq&\; |\xi|^{2s-1} \bigg( 2s\|\widehat f \|_{L^{\infty}(\R)} +\Big\|\frac{d\widehat f(\xi)}{d\xi}  \Big \|_{L^{\infty}(\R)} \bigg)
	 		= c_1' |\xi|^{2s-1} \end{split}\]
	 		and for $|\xi|>1$ 	 		
	 \[\begin{split} \Big| \frac{d}{d\xi} \big( |\xi|^{2s}\widehat f(\xi)\big) \Big| \leq &\; 2s |\xi|^{2s-1} |\widehat f(\xi)| + |\xi|^{2s}\Big|\frac{d}{d\xi} \widehat f (\xi)\Big| \\	
	 			\leq &\; 2s [\widehat f]_{\Sa(\R)}^{0,2} |\xi|^{2s-3} + |\xi|^{2s -3} [\widehat f]_{\Sa(\R)}^{1,3}
	 			\leq c_2' |\xi|^{-1}, \\
	 		\end{split}\]
which proves that $f$ satisfies \eqref{condffff}. From Proposition \ref{proposition:ansss} it follows that
	\begin{equation*}
		\begin{split}
	<\Phi, \frlap f>_s=   &\;\int_{\Rn} \Phi(x) \F^{-1}\Big( (2\pi|\xi|)^{2s} \widehat f(\xi) \Big)(x) \, dx \\
	= &\;\int_{\Rn} (2\pi|\xi|)^{-2s} (2\pi |\xi|)^{2s} {\widehat f (\xi)}  \, d\xi
	=\int_{\Rn} {\widehat f (\xi)}  \, d\xi
	= f(0).
			\end{split}
	 \end{equation*}

Therefore in the distributional sense
\begin{equation*} \frlap \Phi = \delta_0 .	\qedhere\end{equation*}
\end{proof}

{We have the following Lemma.
\begin{lemma}\label{lem:uff1}
Let  $f \in  C_c(\Rn) $, then $ f*\Phi \in  L_s^1(\Rn)$.
\end{lemma} }

\begin{proof}
To prove that $f*\Phi \in  L_s^1(\Rn)$, we suppose that $\mbox{supp}\, f\subseteq B_R$ and we compute
\[\begin{split}
\int_{\Rn} \frac{|f*\Phi(x)|}{1+|x|^{n+2s} }\, dx =&\;
	\int_{B_R}| f(y)|\bigg(\int_{\Rn}\frac{ \Phi(x-y)}{1+|x|^{n+2s}} \, dx \bigg) \, dy\\
	\leq &\; \|f\|_{L^\infty(\Rn)} \int_{B_R} \bigg(\int_{\Rn}\frac{ \Phi(x-y)}{1+|x|^{n+2s}} \, dx \bigg) \, dy.
	\end{split} \]
We set
	\begin{equation} \label{intppr}{ c_{n,s,R} :=\int_{B_R} \bigg(\int_{\Rn}\frac{ \Phi(x-y)}{1+|x|^{n+2s}} \, dx \bigg) \, dy} \end{equation}
	and prove it is a finite quantity.
	We take for simplicity $R=1$ and remark that the constants in the next computations may change value from line to line. For $n > 2s$ we have that
   \[\begin{split}
   	 \int_{B_1} \bigg(\int_{\Rn}\frac{ \Phi(x-y)}{1+|x|^{n+2s}} \, dx \bigg) \, dy  = a(n,s) \int_{B_1} \bigg(\int_{\Rn}\frac{ |x-y|^{2s-n}}{1+|x|^{n+2s}} \, dx \bigg) \, dy.   	 	
   \end{split} \]
 For $x$ small we have that
      \[\begin{split}
       \int_{B_1} \bigg(\int_{B_2}\frac{ |x-y|^{2s-n}}{1+|x|^{n+2s}} \, dx \bigg) \, dy   \leq&\; \int_{B_1} \left( \int_{B_{2}} |x-y|^{2s-n} \, dx\right)\, dy\\
       \leq &\;  c_n \int_{B_1} \bigg(\int_0^{2+|y|} t^{2s-1}  \, dt \bigg) \, dy  \\
       =&\; c_{n,s} \int_0^1 (2+t)^{2s} t^{n-1}\, dt  = \bar c_{n,s}.
   \end{split} \]
For $x$ large, we use that $|x-y|\geq |x|-|y|$ and $1+|x|^{n+2s}>|x|^{n+2s}$, thus
  \[\begin{split}
    \int_{B_1} \bigg(\int_{\Rn \setminus B_2}\frac{ |x-y|^{2s-n}}{1+|x|^{n+2s}} \, dx \bigg) \, dy \leq &\; \int_{B_1} \bigg(\int_{\Rn \setminus B_2} (|x|-|y|)^{2s-n} |x|^{-n-2s} \, dx \bigg) \, dy\\
    =&\; c_n \int_0^1 t^{n-1}\bigg(\int_2^{\infty} (\rho-t)^{2s-n} \rho^{-n-2s} \rho^{n-1} \, d\rho \bigg)\, dt \\
    \leq &\;c_n \int_2^{\infty}  (\rho-1)^{2s-n-1} \, d\rho = c_{n,s}.
   \end{split} \]
   Hence for $n>2s$ the quantity $c_{n,s,R}$ in \eqref{intppr} is finite.	
Meanwhile, for $n<2s$ for $x$ small the same bound as for $n>2s$ holds. For $x$ large, we have that
  \[\begin{split}
    \int_{B_1} \bigg(\int_{\R \setminus B_2}\frac{ |x-y|^{2s-1}}{1+|x|^{1+2s}} \, dx \bigg) \, dy \leq &\; \int_{B_1} \bigg(\int_{\R \setminus B_2} (|x|+|y|)^{2s-1} |x|^{-1-2s} \, dx \bigg) \, dy\\
    =&\; c  \int_0^1 \bigg(\int_2^{\infty} (\rho+t)^{2s-1} \rho^{-1-2s}  \, d\rho \bigg)\, dt  = c_s.\\
     \end{split} \]
In the case $n=2s$ from the triangle inequality we have that
	\[ \int_{B_1} \bigg( \int_{B_2} \frac {\log |x-y|}{1+|x|^2} \, dx \bigg) \, dy  \leq  c  \int_0^2 \log(t+1) \, dt = \tilde c\]
	and
	\[\begin{split}
	\int_{B_1} \bigg(  \int_{\R \setminus B_2}  \frac {\log |x-y|}{1+|x|^2} \, dx \bigg) \, dy   \leq  \int_2^{\infty} \log(t+1) t^{-2}\, dt =\tilde c.
		\end{split}\]
Hence $c_{n,s,R}$ in \eqref{intppr} is finite and we have that
	\begin{equation} \int_{\Rn} \frac{|f*\Phi(x)(x)|}{1+|x|^{n+2s} }\, dx \leq  c_{n,s,R}\|f\|_{L^\infty(\Rn)} \label{db2} .\end{equation} It follows that $f*\Phi\in L_s^1(\Rn)$, as stated.
	\end{proof}
	
Before continuing with the main result of this section, we introduce the following lemma, that will be the main ingredient in the proof of the upcoming Theorem \ref{theorem:poissonsolution}.
\begin{lemma}\label{lem:uff2}
Let  $f \in  C^{\infty}_c(\Rn) $, let $\varphi $ be an arbitrary function such that we have $\wck \varphi \in \Sa_s(\Rn)$ and the following hold:\\
a) for $n>2s$, $\varphi \in L^1(\Rn)\cap C(\Rn)$, \\
b) for $n\leq 2s$, $\varphi \in L^1(\R)\cap C(\R) \cap C^1\big( (-\infty,0)\cup (0,\infty)\big))$ and
	 \begin{equation*} \begin{aligned}
	 &|\varphi(x)|\leq c_1|x|^{2s} \quad &\text{ for } & x \in \R\\
	   & |\varphi(x)|\leq \frac{c_2}{|x|} \quad &\text{ for } &|x|>1 \\
	    &|\varphi'(x)|\leq c'_1|x|^{2s-1} \quad &\text{ for } &0< |x|\leq 1\\
	  &  |\varphi'(x)|\leq \frac{c'_2}{|x|} \quad &\text{ for } &|x|>1. \end{aligned} \end{equation*}
	  Then	
 \begin{equation} \label{disf2} \int_{\Rn} f*\Phi(x) \wck \varphi (x) \, dx = \int_{\Rn} ({2\pi}|x|)^{-2s} \wck f(x) \varphi(x) \, dx.\end{equation}
\end{lemma}
\begin{proof}
In order to prove identity \eqref{disf2} we notice that by the Fubini-Tonelli theorem we have that
	\[\begin{split}
		\int_{\Rn} f*\Phi(x) \wck \varphi(x) dx =&\; \int_{\Rn} \bigg( \int_{\Rn} \Phi(y) f(x-y) \, dy\bigg) \, \wck \varphi(x)\, dx \\
		= &\;\int_{\Rn} \Phi(y)\bigg( \int_{\Rn} f(x-y) \wck\varphi(x)\, dx \bigg) \, dy.  \end{split}\]
		We denote
		\[ f\bar*\wck \varphi (y):=\int_{\Rn} f(x-y)\wck \varphi(x) \, dx =\int_{\Rn} f(x)\wck \varphi(x+y)\, dx\]
		and write \begin{equation} \label{fbarwck1} \int_{\Rn} f*\Phi(x) \wck \varphi(x) dx  =\int_{\Rn} \Phi(y) f\bar*\wck\varphi(y) \, dy.\end{equation}
The operation $\bar* $ is well defined for $f \in C^{\infty}_c(\Rn)$ and $\wck \varphi \in \Sa_s(\Rn)$, furthermore it is easy to see that
	\[ \F(f\bar *\wck\varphi)(x) = \wck f (x)  \varphi (x).  \] We notice at first that since $\varphi$ and $\wck \varphi$ are continuous, $\F(\wck \varphi)=\varphi$ on $\Rn$. We define
		\begin{equation}\label{psifi2}  \psi(x) := \F(f\bar *\wck\varphi)(x) = \wck f (x)  \varphi (x)\end{equation} and we write \eqref{fbarwck1} as
		\begin{equation}\label{psifi1}  \int_{\Rn} f*\Phi(x) \wck \varphi(x) dx  =\int_{\Rn} \Phi(y) \wck \psi(y) \, dy .\end{equation}
		To apply Proposition \ref{proposition:ansss}, we have to check that $\psi$ verifies \eqref{condffff}.
Since \[ \int_{\Rn} |\psi(x)|\, dx = \int_{\Rn} |\wck f(x)| |\varphi(x)| \, dx\leq \|\wck f\|_{L^{\infty}(\Rn)} \|\varphi\|_{L^1(\Rn)},\] we have that $\psi \in L^1(\Rn)$. Also, $\psi \in C(\Rn)$ as a product of continuous functions.  We claim that $f\bar*\wck\varphi \in \Sa_s(\Rn)$. Indeed, suppose $\text{supp}\, f\subseteq B_R$ for $R>0$. We remark that in the next computations the constants may change from line to line. Then for $|x|\leq 2R$ we have that	
	\[ \begin{split}
			(1+|x|^{n+2s} ) |f\bar*\wck\varphi(x)| \leq &\; c_{n,s,R} \int_{B_{R}(x)} |f(y-x) \wck\varphi(y)|\, dy \\
			\leq &\;  c_{n,s,R} \|f\|_{L^\infty(B_R)} \|\wck \varphi\|_{L^{\infty}(B_{3R})}.
			\end{split}\]
For $|x|>2R$ we have that
	\[ \begin{split} |x|^{n+2s} |f\bar*\wck\varphi(x)|\leq\|f\|_{L^\infty(B_R)} [\wck \varphi]_{\Sa_s(\Rn)}^0 |x|^{n+2s} \int_{B_{R}  } |x+y|^{-n-2s} \, dy
	\end{split}\]
and we remark that $|y| \leq |x| /2$ (otherwise $y\notin \text{supp} \,f$). Then we use the bound $|x+y|\geq |x|-|y|\geq |x|/2$ and we have that
	\[
	 |x|^{n+2s} |f\bar*\wck\varphi(x)|\leq\|f\|_{\infty} [\wck \varphi]_{\Sa_s(\Rn)}^0  |x|^{n+2s}  \int_{ B_{R} }  |x|^{-n-2s}\, dy =c_{n,s,R}.
	\]
	We can iterate the same method to prove that $(1+|x|^{n+2s} ) |D^{\alpha} f\bar*\wck\varphi(x)|$ is bounded since $D^{\alpha} f\bar*\wck\varphi(x)=f\bar*D^{\alpha} \wck\varphi(x)$ and $D^{\alpha} \wck\varphi \in \Sa_s(\Rn)$.
	For $n\leq 2s$ we have that 	
\[\begin{aligned}
	 		&  |\psi(x)| \leq | \wck f(x)||\varphi(x)| \leq \|\wck f\|_{L^{\infty}(\R)} c_1|x|^{2s} \quad &\text{ for } &|x| \leq 1,\\
	 		&  |\psi(x)| \leq | \wck f(x)||\varphi(x)| \leq \|\wck f\|_{L^{\infty}(\R)} \frac{c_2}{|x|} \quad &\text{ for } &|x| > 1 .\end{aligned}\]
	 		Moreover, for $|x|>1$
	 		\[\begin{split}
	 		|\psi'(x)| \leq&\; | \wck f(x)||\varphi'(x)| + \Big|\frac{d}{dx} \wck f (x) \varphi(x)\Big|\\
	 		 \leq &\;\|\wck f\|_{L^{\infty}(\R)} \frac{c'_2}{|x|} + \Big| \int  f (\xi) (i\xi)e^{ix\xi} \, d\xi\Big| |\varphi(x)|\\
	 		 \leq&\; \|\wck f\|_{L^{\infty}(\R)} \frac{c'_2}{|x|} + \|\xi f(\xi)\|_{L^1(\R)} \frac{c_2}{|x|} 	 \leq\frac{C}{|x|}
	 		\end{split}\]
and for $|x|\leq 1$, since $f\in C_c^{\infty}(\R)$
 	\[\begin{split}
	 		|\psi'(x)| \leq&\; | \wck f(x)||\varphi'(x)| + \Big|\frac{d}{dx} \wck f (x)\Big| | \varphi(x)|\\
	 					\leq&\; \|\wck f\|_{L^{\infty}(\R)} c_1' |x|^{2s-1} + c_1|x|^{2s} \|\xi f(\xi)\|_{L^{\infty}(\R)} |x|^{-1}
	 					= C|x|^{2s-1}.
\end{split}\]
Hence $\psi$  satisfies \eqref{condffff}. Taking into account \eqref{psifi1} and applying Proposition \ref{proposition:ansss} we have that
			\[ \int_{\Rn} f*\Phi(x)\wck \varphi(x)\, dx = \int_{\Rn} \Phi (x) \wck \psi (x) \, dx =\int_{\Rn} ({2\pi}|x|)^{-2s} \psi (x)  \, dx,\]
			and from \eqref{psifi2} we conclude that
				\[  \int_{\Rn} f*\Phi(x)\wck \varphi(x)\, dx=\int_{\Rn} ({2\pi}|x|)^{-2s}  \wck f (x) \varphi(x)\, dx. \qedhere\]
\end{proof}
The function $\Phi$ gives the representation formula for equation $\frlap u=f$ both in the distributional sense and pointwise.

\begin{theorem}
\label{theorem:poissonsolution}
 Let $f \in  {C^{2s+\eee}_c(\Rn)} $ and let $u$  be defined as
	\[  u(x):=\Phi*f\, (x) .\]
Then $ u\in L_s^1(\Rn)$ and in the distributional sense
	\[ \frlap u =f. \]
	Moreover, $u  \in C^{2s+\eee}(\Rn)$ and pointwise in $\Rn$ \[ \frlap u(x)=f(x).\]
\end{theorem}
\begin{proof}

From Lemma \ref{lem:uff1}, we have that $u\in L_s^1(\Rn)$. We prove at first the statement for $f\in C^{\infty}_c(\Rn)$.

We notice that for $\varphi \in \mathcal{S}(\Rn)$, the function $\mathcal F^{-1} ((2\pi |\xi)|^{2s} \widehat \varphi (\xi)) \in \mathcal{S}_s(\Rn)$ and $(2\pi |\xi)|^{2s} \widehat \varphi (\xi)$ satisfies the hypothesis of Lemma \ref{lem:uff2}. Hence, by \eqref{disf2}
	\begin{equation*}
		\begin{split}
	<u(x), \frlap \varphi>_s =&\;\int_{\Rn}  f*\Phi(x) \mathcal{F}^{-1} \Big((2\pi |\xi|)^{2s}\widehat \varphi(\xi) \Big)(x) \, dx\\
		=&\; \int_{\Rn} \wck{f}(\xi) {\widehat \varphi(\xi)} \, d\xi=  \int_{\Rn}  f(x) \varphi(x) \, dx.
	\end{split}
	\end{equation*}
	The last equality follows since $\wck f \in L^1(\Rn)$, which is assured by the infinite differentiability of $f$.
	We conclude that $u$ is the distributional solution of \[\frlap u = f.\]

We consider now $f\in C^{2s+\eee}_c(\Rn) $.
We take a sequence of functions $(f_k)_k \in C^{\infty}_c (\Rn)$ such that $\|f_k-f\|_{L^{\infty}(\Rn)} \underset{k\to \infty}{\longrightarrow} 0$ and we consider $u_k=\Phi*f_k$. Then we have that for any $\varphi \in \Sa(\Rn)$
	\[ <u_k, \frlap \varphi>_s= \int_{\Rn} f_k (x) \varphi(x)\, dx.\] By definition of $f_k$ \[ \lim_{k \to +\infty} \int_{\Rn} f_k(x)\varphi(x)\, dx = \int_{\Rn} f(x)\varphi(x)\, dx ,\]
	moreover, using \eqref{db1} and \eqref{db2} we have that
		\[\begin{split}
		 <u_k -u,\frlap \varphi>_s \leq&\; [\frlap \varphi]^{0}_{\Sa_s(\Rn)}\|u_k-u\|_{L_s^1(\Rn)}\\
		  \leq&\;  c_{n,s,R}[\frlap \varphi]^{0}_{\Sa_s(\Rn)}\|f_k-f\|_{L^{\infty}(\Rn)} \underset{k \to \infty}{\longrightarrow} 0.
		 \end{split}\]
	We thus obtain that for any $\varphi \in \Sa(\Rn)$ \[<u,\frlap \varphi>_s =\int_{\Rn} f(x)\varphi(x)\, dx.\] Hence in the distributional sense $\frlap u = f $ on $\Rn$ for any $f\in C^{2s+\eee}_c(\Rn)$.
	{To obtain the pointwise solution, we notice that thanks to Theorem 9.3 in \cite{Zygmund} we have that $u\in C^{2s+\eee}(\Rn)$. This, together with the fact that $u\in L_s^1(\mathbb{R}^n)$ implies that the fractional Laplacian of $u$ is well defined, according to Remark  \ref{blacaz}}. Moreover, from the continuity of the mapping $ \Rn \ni x \mapsto \frlap u(x) $, according to Proposition 2.1.7 from \cite{Silvestre}, we have that
		$  \int_{\Rn}\frlap u(x) \varphi(x)\, dx$ is well defined. For any $\varphi \in C^{\infty}_c(\Rn)$ we have that
		\[\int_{\Rn} u(x)\frlap \varphi(x)\, dx = \int_{\Rn} f(x)\varphi(x)\, dx .\]  Thanks to Fubini-Tonelli's Theorem and changing variables we obtain that for any $\varphi \in C^\infty_c(\Rn)$
			\[ \int_{\Rn} f(x)\varphi(x)\, dx=\int_{\Rn} u(x)\frlap \varphi(x)\, dx = \int_{\Rn} \frlap u(x)\varphi(x) \, dx  .\]  Since both $f$ and $\frlap u$ are continuous, we conclude that pointwise in $\Rn$ \[ \frlap  u(x) = f(x).\qedhere \]
 \end{proof}

{As a corollary, we have a representation formula for a $C^{\infty}_c ({\Rn})$ function.}
\begin{corollary}
\label{lemma:unouno}
For any $f \in C^{\infty}_c ({\Rn})$ there exists a function $\varphi \in C^{\infty}({\Rn})$ such that 		
	\[f(x)=\varphi* \Phi(x), \]  and $\varphi (x)=\mathcal{O} (|x|^{-n-2s})$ as $|x| \rightarrow \infty$.
\end{corollary}

\begin{proof}[Proof]
For $f \in C^{\infty}_c(\Rn)$, we define $\varphi$ as
	\[ \varphi (x):=\frlap f(x).\]
The bound established in \eqref{frb1} assures the asymptotic behavior of~$\varphi$, while it is not hard to see that $\varphi \in C^{\infty}({\Rn})$.
Then by using Theorem \ref{theorem:poissonsolution} we have that pointwise in $\Rn$
	\[\varphi*\Phi(x)= \frlap f *\Phi(x)=  f(x).\qedhere\]

\end{proof}

\subsection{The Poisson kernel} \label{thepoissonkernel}

We claim that $P_r$ plays the role of the fractional Poisson kernel. Indeed, the function $P_r$ arises in the construction of the solution to Dirichlet problem with vanishing Laplacian inside the ball and a known forcing term outside the ball, as stated in the next Theorem~\ref{theorem:DPL}.
\begin{theorem}
\label{theorem:DPL}
 Let $r>0$, $g \in L^1_s(\Rn) \cap C({\Rn})$ and let
\begin{equation}
		 u_g(x) : =
			\begin{cases}	
				\displaystyle  \int_{{\Rn}\setminus B_r} P_r(y,x) g(y)\, dy &\quad  \, \text{if } x\in B_r, \\
				g(x) &\quad \, \text{if } x \in {\obal}.
			\end{cases} \label{solD}
	\end{equation}
Then $u_g$ is the unique pointwise continuous solution of the problem \eqref{LaplaceeqD}
	\begin{equation*}
	\begin{cases}
	\frlap u= 0 \qquad  &\mbox{ in }  {B_r},
\\	u= g \qquad  &\mbox{ in }  {\obal}.
		\end{cases}
	\end{equation*}
\end{theorem}
\begin{proof}[Proof of Theorem \ref{theorem:DPL}]
We see at first that $u_g \in L_s^1(\Rn) $. Take $R>2r$ and $x \in B_r$, then by using \eqref{Ip}, the inequality $|x-y| >|y|-r$ and for $|y|>R$ the bound
	\begin{equation} \label{bbly1}\frac{|y|^{n+2s}}{(|y|^2-r^2)^s|x-y|^n} \leq  2^{n+s}\end{equation}
we have that
	\[ \begin{split} |u_g(x)|\leq &\;  \int_{R>|y|>r} P_r(y,x) |g(y)|\, dy + \int_{|y|>R} P_r(y,x) |g(y)|\, dy\\
							\leq&\; c(n,s) \sup_{y\in \overline B_R\setminus B_r} |g(y)| + 2^{n+s} c(n,s) (r^2-|x|^2)^s \int_{|y|>R}  \frac{|g(y)|}{|y|^{n+2s}}\, dy\\
							\leq&\; c(n,s) \sup_{y\in \overline B_R\setminus B_r} |g(y)| + 2^{n+s} c(n,s) r^{2s}  \int_{|y|>R}  \frac{|g(y)|}{|y|^{n+2s}}\, dy.
							 \end{split} \]
Since $g \in L_s^1(\Rn)$, the last integral is bounded, and so $u_g$ is bounded in $B_r$. It follows that $u_g\in L_s^1(\Rn)$, as stated. Moreover, the local $C^\infty$ regularity of $u_g$ in $B_r$ follows from the regularity of the Poisson kernel.

Let us fix $x \in B_r$ and prove that $u_g$ has the $s$-mean value property in $x$. If this holds, indeed, Theorem \ref{theorem:arm} implies that $\frlap u(x)=0$, and given the arbitrary choice of $x$, the same is true in the whole $B_r$.


We claim that for any $\rho$ such that $0<\rho< r-|x|$ we have
\begin{equation} A_{\rho} * u_g (x) = u_g(x).\label{claim} \end{equation}

Let at first $g$ be in $C^{\infty}_c{(\Rn)} $. By Corollary \ref{lemma:unouno}, there exists a function $\varphi\in C^{\infty}(\Rn)$ such that
	  \[ g(y)= \int_{\Rn} \Phi(z-y) \varphi(z) \, dz \] and at infinity $\varphi (z) = \mathcal{O}(|z|^{-n-2s})$.
For $r>0$ fixed, we write $g$ as
	\begin{equation} \label{bbbb} g(y)=\int_{\obal} \Phi(t-y) \varphi(t) \, dt + \int_{B_r} \Phi(z-y) \varphi(z) \, dz.\end{equation}
Using identity \eqref{Ipu} we have that
	\begin{equation*}
		\begin{split}
	\int_{B_r} \Phi(z-y) \varphi(z) \,  dz  =\;&\int_{B_r} \bigg(\int_{\obal}P_r(t,z)\Phi(y-t) \, dt\bigg)\, \varphi(z) \, dz\\
	= \;&\int_{\obal}\Phi(y-t) \bigg(\int_{B_r}P_r(t,z)\varphi(z)\, dz\bigg)\, dt.
		\end{split}
	\end{equation*}
Therefore, in \eqref{bbbb} it follows that
	 \begin{equation} g(y) = \int_{\obal} \Phi(y-t)\psi(t)\, dt, \label{gsecondcase}\end{equation}
where $\psi(t) = \varphi(t)+\int_{B_r}P_r(t,z)\varphi(z) \, dz$.
In particular, using \eqref{solD} and \eqref{gsecondcase} we have that {
	\begin{equation*} \begin{split}
		 u_g(x)=&\; \int_{|y|>r}P_r(y,x) \bigg(\int_{|t|>r} \Phi(y-t) \psi(t) \, dt\bigg) \, dy\\
		 	= &\; \int_{|t|>r} \psi(t)\bigg( \int_{|y|>r} P_r(y,x) \Phi(y-t) \, dy\bigg) \, dt = \int_{|t|>r} \psi(t) \Phi(x-t) \, dt\end{split}
		 \end{equation*}
thanks to \eqref{Ipu}. Furthermore, we compute
	\begin{equation*} \begin{split}
		A_{\rho}*u_g(x) =&\; \int_{|y|>\rho} A_\rho(y)  \bigg(\int_{|t|>r} \psi(t) \Phi(x-y-t) \, dt \bigg)\, dy  \\
					=&\;    \int_{|t|>r} \psi(t) \bigg(\int_{|y|>\rho} A_\rho(y) \Phi(x-y-t) \, dy\bigg) \, dt.					
	\end{split}
		 \end{equation*}}
Having chosen $\rho\leq r-|x|$ we have that $|x-t|\geq |t|-|x| \geq \rho$ and from \eqref{Ifu} we obtain
 \begin{equation*}
		A_{\rho}*u_g(x) = \int_{|t|>r} \psi(t)  \Phi(x-t)\, dt.
	 \end{equation*}
Consequently $A_{\rho} * u_g (x) =u_g(x),$ thus for $g\in C^{\infty}_c({\Rn})$ the claim \eqref{claim} is proved.\\

We now prove the claim \eqref{claim} for any forcing term $g \in L^1_s({\Rn})\cap C(\Rn)$. In particular, let $\eta_k \in C^{\infty}_c(\Rn)$ be such that $\eta_k(x) \in [0,1]$, $\eta_k=1$ in $B_k$ and $\eta_k=0$ in $B_{k+1}$. Then $g_k:=\eta_k g \in C^{\infty}_c(\Rn)$ and we have that $g_k \underset{k\to \infty}{\longrightarrow} g$ pointwise in $\Rn$, in norm $L_s^1(\Rn)$ and uniformly on compact sets. So, for any $k \geq 0$ the function $u_{g_k}(x)$ has the $s$-mean value property in $x$. Precisely, for any $\rho>0$ small independent of $k$,
	\begin{equation} \big(A_{\rho}*u_{g_k}\big)(x) =u_{g_k}(x). \label{primacosa} \end{equation}
We claim that
	\begin{equation}
		\lim_{k\to \infty} u_{g_k}(x) = u_g(x) \label{puno}
	\end{equation}
and that for any $\rho>0$ small
	\begin{equation}
		\begin{split}
		&\lim_{k\to \infty} \big(A_{\rho}*u_{g_k}\big)(x) = A_{\rho}*u_g (x)  \label{pdue}.
		\end{split}
	\end{equation}
Let $\eee $ be any arbitrarily small quantity. For $k$ large and $R>2r,$ we take advantage of \eqref{bbly1} and obtain that for $x \in B_r$
\[ \begin{split}	
		|u_{g_k} (x) -u_g(x)| \leq &\; \int_{\Rn\setminus B_r} |g_k(y)-g(y)| P_r(y,x)\,dy \\
				\leq&\;  2^{n+s} c(n,s) (r^2-|x|^2)^s  \int_{\Rn \setminus B_R}  \frac{|g_k(y)-g(y)| }{|y|^{n+2s}} \, dy \\
				&\;+ \sup _{y \in \overline B_R\setminus B_r}|g_k(y)-g(y)| \int_{B_R\setminus B_r}  P_r(y,x)\, dy\\	
				\leq &\; c(n,s,r) \int_{\Rn \setminus B_R}  \frac{|g_k(y)-g(y)| }{|y|^{n+2s}} \, dy + \sup _{y \in \overline B_R\setminus B_r}|g_k(y)-g(y)|  \leq \eee \end{split}\]
by the convergence in $L_s^1(\Rn)$ norm, the uniform convergence on compact sets of $g_k$ to $g$ and integrability in $\Rn \setminus B_r$ of the Poisson kernel (by identity ~\eqref{Ip}). Hence, claim \eqref{puno} is proved.			
In order to prove claim \eqref{pdue}, we notice that for any $\rho>0$ small we have that
	\begin{align} \label{illy}
		 | A_{\rho}* u_{g_k}(x)  - A_\rho * u_g(x) |\leq &  \int_{|y|>\rho} A_{\rho}(y)|   u_{g_k}  (x-y)  -u_g(x-y)| \, dy \notag\\
		\leq & \int_{\substack{ {|y|>\rho}\\{|x-y|\geq r}}} A_{\rho}(y) |g_k(x-y)-g(x-y)|  \, dy \notag\\
		&+  \int_{|z|>r}  | g_k(z)-g(z)|    \int_{\substack{ {|y|>\rho}\\ {|x-y|<r}}}    A_{\rho} (y) P_r(z,x-y)  dy \, dz\notag\\
		=&\; I_1+I_2.
	\end{align}
Let $R>2\rho$. Thanks to the bound \eqref{bbly1} for $|y|>R$, the convergence in $L_s^1(\Rn)$ norm, the uniform convergence on compact sets of $g_k$ to $g$ and the integrability in $\Rn \setminus B_{\rho}$ of the $s$-mean kernel (by identity \eqref{Ir}) we have that for $k$ large
\[	\begin{split}
		I_1 =&\; c(n,s) r^{2s} \int_{\substack{ {|y|>\rho}\\{|x-y|\geq r}}} \frac{|g_k(x-y)-g(x-y)|} {(|y|^2-\rho^2)^s |y|^n} \, dy\\
			\leq &\;2^{n+s} c(n,s,r) \int_{|y|>R} \frac{|g_k(x-y) - g(x-y)| }{|y|^{n+2s}}\, dy \\
						&\; + \sup_{y \in \overline B_R\setminus B_\rho} |g_k (x-y)-g(x-y)| \int_{R>|y|>\rho} A_\rho(y) \, dy
						\leq  \frac{\eee}{2}. \end{split} \]
Once more, for $R>2r$ and $|z|>R$ we use the bound \eqref{bbly1} and we have that
\[	\begin{split}
		I_2 = &\; \int_{|z|>R} |g_k(z)-g(z)| \int_{\substack{{|y|>\rho}\\ {|x-y|<r}}} A_\rho(y)P_r(z,x-y)  dy   \, dz \\\
		&\;  +  \int_{R>|z|>r} |g_k(z)-g(z)| \int_{\substack{{|y|>\rho}\\ {|x-y|<r}}} A_\rho(y)P_r(z,x-y)  \, dy \, dz \\
	\leq &\;c(n,s) \int_{\substack{{|y|>\rho}\\ {|x\!-\!y|<r}}} A_\rho(y) (r^2\!-\!|x\!-\!y|^2)^s \!   \int_{|z |>R} \frac{|g_k(z) \!-\! g(z)| }{(|z|^2\! -\!r^2)^s |z\!-\!x\!+\!y|^n}\, dz  \, dy \\
						&\;+  \sup_{z\in \overline B_R\setminus B_r} |g_k (z)-g(z)|  \int_{\substack{{|y|>\rho}\\ {|x-y|<r}}} A_\rho(y)   \int_{R>|z|>r} P_r(z,x-y)  \, dz  \, dy \\
				\leq &\;c(n,s,r)  \int_{|z|>R} \frac{|g_k(z) - g(z)| }{|z|^{n+2s}}dz   + \sup_{z \in \overline B_R\setminus B_r} |g_k (z)-g(z)| 
				\end{split} \]
since by identity \eqref{Ir} and \eqref{Ip}
	\[
	\int_{\substack{{|y|>\rho}\\ {|x-y|<r}}} A_\rho(y)   \int_{R>|z|>r} P_r(z,y-x)  \, dz  \, dy  	\leq 1.
	\]
 Therefore again by the convergence in $L_s^1(\Rn)$ norm, the uniform convergence on compact sets of $g_k$ we have that $I_2 \leq 	\displaystyle \frac{\eee}{2}$. In \eqref{illy} it follows that
\begin{equation*}
		\lim_{k\to \infty} \big(A_{\rho}* u_{g_k}\big)(x) =  A_{\rho}*u_g(x),
	\end{equation*}
thus the desired result \eqref{pdue}.

By \eqref{primacosa}, \eqref{puno} and \eqref{pdue} we have that
\[A_{\rho}*u_g (x) = \lim_{k \to \infty}A_{\rho}* u_{g_k}(x)   = \lim_{k \to \infty} u_{g_k} (x) =u_g(x),\]
thus $u_g$ has the $s$-mean value property at $x$.
This concludes the proof of the claim \eqref{claim}.

We now prove the continuity of $u_g$. Of course, $u_g$ is continuous in $B_r$ and in $\obal$. We need to check the continuity at the boundary of $B_r$.

Let $y_0 \in \partial B_r$ and $\eee>0$ arbitrarily small to be fixed,  $\delta_{\eee}>0$ be such that, if $y \in B_{\delta_{\eee}}(y_0)$ then $|g(y)-g(y_0)| < \eee$. We fix $\mu$ arbitrarily small such that $0 <\mu<  \frac{\delta_{\eee}}{2}$, $R> 2r$, and $x \in  B_r \cap \ B_{\mu} (y_0 )$. Notice that \[ r^2-|x|^2 = (r+|x|) (r-|x|) < 2r |y_0-x| < 2r \mu.\]
From \eqref{Ip} we have that
	\begin{equation}\label{illy1} |u_g(x)-g(y_0)| \leq \int_{\obal} \big|g(y)-g(y_0)\big|P_r(y,x) \, dy .\end{equation}
For $r<|y|<R$ and $|y-y_0|\geq \delta_\eee$ we have that
	$ |x-y| \geq \delta_{\eee} - \mu>  \frac{\delta_{\eee}}{2} .$
Meanwhile, for $|y|>R$ we use the bound \eqref{bbly1}. We have that
	\begin{equation*}
		\begin{split}
		& \int_{\substack{ {|y|>r}\\{|y-y_0|\geq  \delta_{\eee}}}}  |g(y)-g(y_0)|P_r(y,x) \, dy\\
		 \leq\; &   c(n,s,R) \mu^s \bigg( \frac{2^n}{ \delta_{\eee}^{n}} \int_{\substack{ {R>|y|>r}\\{|y-y_0|\geq  \delta_{\eee}}}} \frac{|g(y)|+|g(y_0)|}{ (|y|^2-r^2)^s }\, dy
		 + 2^{n+s} \int_{|y|>R} \frac{|g(y)|+ |g(y_0)|}{ |y|^{n+2s}}\, dy   \bigg) \\
	\leq \; &  c(n,s,R) \mu^s \bigg(\frac{2^n }{\delta_{\eee}^{n}} \overline c(r,R,s,g)
	+ 2^{n+s} \|g\|_{L_s^1(\Rn)}+ \tilde c (g,s, R)  \bigg) \\
		=  &\;C(n,s,R,r,g,\delta_\eee) \mu^s .
		\end{split}
	\end{equation*}
	From this and the fact that	\[\int_{\substack{  {|y|>r}\\{|y-y_0|< \delta_{\eee}}} } |g(y)-g(y_0)|P_r(y,x) \, dy \leq \eee \int_{\obal} P_r(y,x)\, dy =\eee\]
		by \eqref{Ip} and the continuity of $g$, we can pass to the limit in \eqref{illy1}. Sending first $\mu  \rightarrow 0$ and afterwards $\eee \rightarrow 0$ we obtain that
	\[  \lim_{x \rightarrow y_0 }\big(u_g(x)-g(y_0)\big) =0,\] thus the continuity of $u_g$.

{The uniqueness of the solution follows from the Maximum Principle. Indeed, if one takes $u_1$ and $u_2$ two different continuous solutions of the Dirichlet problem, then $u=u_1-u_2$ is a continuous solution to the problem
	 \begin{align*}
		\frlap u(x)  &=0, &  &\text{in } B_r \\
		u(x)& =0, & &\text{in } \Rn \setminus  B_r.
		\end{align*}
By Theorem 3.3.3 in \cite{nonlocal}, the solution $u$ is constant, hence null since it is continuous in $\Rn$ and vanishing outside of $B_r$.} This concludes the proof of the Theorem.
\end{proof}

\section[The Green function for the ball]{The Green function for the ball}\label{green}

The purpose of this section is to prove Theorems~\ref{theorem:thm1} and~\ref{theorem:thm2}.
Theorem~\ref{theorem:thm1} introduces a formula for the Green function on the ball, that is more suitable for applications. In Theorem \ref{theorem:thm2} the solution to the Dirichlet problem with vanishing data outside the ball and a given term inside the ball is built in terms of the Green function. We also compute the normalization constants needed in the formula of the Green function on the ball.

\subsection{Formula of the Green function for the ball}\label{thm1}
This subsection focuses on the proof of Theorem \ref{theorem:thm1}, that we recall here.
\begin{theorem}
\label{theorem:thm1}
Let $r>0$ be fixed and let $G$ be the function defined in \eqref{greendefn}. Then if $n \neq 2s$
\[  G(x,z) =  \kappa(n,s) |z-x|^{2s-n}  \int_0^{r_0(x,z)}  \frac{t^{s-1}} {(t+1)^\frac{n}{2}} \, dt ,\]
where
 	\begin{equation} \displaystyle r_0(x,z) = \frac{(r^2-|x|^2)(r^2-|z|^2)}{r^2|x-z|^2}\label{ro} \end{equation}
and $\kappa(n,s)$ is a constant depending only on $n$ and $s$.\\
For $n=2s$, the following holds
\begin{equation} G(x,z)= \kappa\Big(1,\frac{1}{2}\Big) \log\bigg( \frac{r^2-xz+\sqrt{(r^2-x^2)(r^2-z^2)}}{r|z-x|}\bigg).\label{formn1s12} \end{equation}
\end{theorem}

We consider the three cases $n>2s$, $n<2s$ and $n=2s$ separately.
\begin{proof}[Proof of Theorem \ref{theorem:thm1} for~$n>2s$]

Let $x,\, z\in B_r$ be fixed.\\
We insert the explicit formula \eqref{fundsolution} into definition \eqref{greendefn} and obtain that
	\begin{equation}G(x,z) =  a(n,s) \big( |z-x|^{2s-n}-A(x,z)  \big), \label{axz}\end{equation}
where \[A(x,z) := \int_{|y|>r} \frac{P_r(y,x) }{|y-z|^{n-2s}} \, dy.\]
Inserting also definition \eqref{poissondefn} we have that
	\begin{equation*}
\begin{split}
 A(x,z)= c(n,s)\int_{|y|>r} \frac{ (r^2-|x|^2)^s} {|y-z|^{n-2s} (|y|^2-r^2)^s|y-x|^n} \, dy .\\
		\end{split}
	\end{equation*}

%

We use the point inversion transformation that is detailed in the Appendix. { Let $x^* \in \Rn \setminus \overline{B}_r$ and $y^* \in B_r$ be the inversion of $x$, respectively $y$ with center at $z$, defined by the relation \eqref{tr}. }
With this transformation, using formulas \eqref{firsttr}, \eqref{dxtr} and \eqref{sectr} we obtain that
	\begin{equation*}
		\begin{split}
		A(x, z)
		=   c(n,s) |z-x|^{2s-n} \int_{B_r}  \frac{  (|x^*|^2-r^2)^s }{  (r^2-|y^*|^2)^s |x^*-y^*|^n } \, dy^*.
		\end{split}	
	\end{equation*}

{We continue the proof for $n>3$. However, the results hold for $n\leq 3$ and can be proved with similar computations. We use hyperspherical coordinates with $\rho>0$ and  $\theta, \theta_1, \dots, \theta_{n-3} \in [0, \pi], \theta_{n-2} \in [0,2\pi]$ (see \eqref{chofvarhypersph} in the Appendix and observations therein).}
Without loss of generality and up to rotations, we assume that $x^* =|x^*|e_n$, so we have the identity $ |x^*-y^*|^2 = \rho^2+|x^*|^2-2|x^*| \rho\cos\theta$ (see Figure~\ref{fign:xycalc} in the Appendix for clarity). With this change of coordinates, we obtain
\begin{equation*}
		\begin{split}
		  A(x,z)=\; &c(n,s)|z-x|^{2s-n} (|x^*|^2-r^2)^s 2\pi \prod_{k=1}^{n-3}\int_0^{\pi} \sin^k\theta\, d\theta  \\
		 & \;  \int_0^r \frac{\rho^{n-1}}{(r^2 -\rho^2)^s}   \bigg(\int_0^{\pi} \frac{ \sin^{n-2}\theta}{(\rho^2+|x^*|^2-2|x^*| \rho\cos\theta)^{n/2}} \, d\theta \bigg) \,d\rho.		
    \end{split}
	\end{equation*}

Let $\tau:=  \frac{|x^*|}{\rho}$ (notice that $\tau >1$). We have that
	\[	\int_0^{\pi} \frac{ \sin^{n-2}\theta}{(\rho^2+|x^*|^2-2|x^*| \rho\cos\theta)^{n/2}} \, d\theta =  \frac{1}{\rho^n}  \int_0^{\pi} \frac{ \sin^{n-2}\theta}{(\tau^2+1-2\tau \cos\theta)^{n/2}} \, d\theta.\]
Thanks to identity \eqref{prop1} we obtain that
	\begin{equation*}
		\begin{split}
\int_0^{\pi} \frac{ \sin^{n-2}\theta}{(\rho^2+|x^*|^2-2|x^*| \rho\cos\theta)^{n/2}} \, d\theta = \; &\frac{1}{\rho^n}  \frac{1}{\tau^{n-2}(\tau^2-1)} \int_0^{\pi} \sin^{n-2}\alpha \, d\alpha \\
=\; &  \frac{1}{|x^*|^{n-2}(|x^*|^2-\rho^2)}\int_0^{\pi}\sin^{n-2}\alpha \, d\alpha.
		\end{split}
	\end{equation*}
Then, using identity \eqref{prop2} and inserting the explicit value of $c(n,s)$ given by \eqref{ctcns}, we arrive at
	\begin{equation} \label{axzz}
		\begin{split}
	 A(x,z)
	= \; & \frac{\sin (\pi s) }{\pi}   |z-x|^{2s-n} \frac{\big(|x^*|^2-r^2\big)^s} {|x^*|^{n-2}}  \int_0^r  \frac{2\rho^{n-1}}{(r^2 -\rho^2)^s (|x^*|^2-\rho^2)}\, d\rho\\
	=  \; & \frac{\sin (\pi s) }{\pi}  |z-x|^{2s-n} \frac{\big(|x^*|^2-r^2\big)^s} {|x^*|^{n-2}} J(x^*),
		\end{split}
	\end{equation}
where \[ J(x^*)=  \int_0^r  \frac{2\rho^{n-1}}{(r^2 -\rho^2)^s (|x^*|^2-\rho^2)}\, d\rho.\]

Now we define the constant \begin{equation} k(n,s) :=  \frac{1}{2}  \bigg(\int_0^1  \tau^{n-2s-1} (1-\tau^2)^{s-1} \, d\tau\bigg)^{-1} \label{intkns}\end{equation}
(we compute its explicit value in Subsection \ref{constants}).
Then we have that
	\[J(x^*)= 2 k(n,s) \int_0^r\frac{2\rho^{n-1}}{(r^2-\rho^2)^s(|x^*|^2-\rho^2)} \bigg(\int_0^1 \tau^{n-2s-1} (1-\tau^2)^{s-1} \, d\tau\bigg)  \, d\rho.\]
We perform the change of variables $t=\tau \rho$ and apply the Fubini-Tonelli's theorem to obtain that
	\begin{equation*}
	\begin{split}
	J(x^*)=\;& 2 k(n,s) \int_0^r \frac{2\rho}{(r^2-\rho^2)^s(|x^*|^2-\rho^2)} \bigg(\int_0^{\rho} t^{n-2s-1} (\rho^2-t^2)^{s-1} \, dt  \bigg)  \, d\rho\\
	=\;& 2k(n,s) \int_0^r t^{n-2s-1} \bigg( \int_t^r \frac{ 2\rho (\rho^2-t^2)^{s-1} }{(r^2-\rho^2)^s(|x^*|^2-\rho^2) }\, d\rho \bigg) \, dt.
	\end{split}
	\end{equation*}
We change variables $\rho^2-t^2=\tau$ and $r^2-\tau-t^2=\rho$ to obtain
\begin{equation*}
	\begin{split}
J(x^*)=\;& 2   k(n,s) \int_0^r t^{n-2s-1} \bigg( \int_0^{r^2-t^2} \frac{\tau^{s-1}}{(r^2-\tau-t^2)^s(|x^*|^2- \tau -t^2)} \, d\tau \bigg) \, dt\\
		 = \;& 2k(n,s)\int_0^r t^{n-2s-1} \bigg( \int_0^{r^2-t^2} \frac{ (r^2-t^2-\rho)^{s-1}}{\rho^s (|x^*|^2+\rho -r^2)} \, d\rho \bigg) \, dt \\
	= \;&  2k(n,s) \int_0^r t^{n-2s-1}  I(t) \, dt,
		\end{split}
	\end{equation*}
where \[I(t) = \int_0^{r^2-t^2} \frac{ (r^2-t^2-\rho)^{s-1}}{\rho^s (|x^*|^2+\rho -r^2)} \, d\rho.\]
Using Proposition \ref{proposition:uss} for $\alpha=r^2-t^2$ and $\beta=|x^*|^2-r^2$
we have that \[  I(t)= \frac{\pi}{\sin(\pi s)} \frac{(|x^*|^2-t^2)^{s-1}}{(|x^*|^2-r^2)^s}.\]
Hence in $J(x^*)$, with the changes of variables $ \frac{|x^*|}{t}= \tau$ and then $\tau^2-1=t$ we have that
	\begin{equation*}
		\begin{split}
		 J(x^*) = \;& 2k(n,s)  \frac{\pi}{\sin(\pi s)} (|x^*|^2-r^2)^{-s} \int_0^r t^{n-2s-1} (|x^*|^2 -t^2)^{s-1} \, dt.\\
		=\;& 2 k(n,s) \frac{\pi}{\sin(\pi s)} \frac{ |x^*|^{n-2}}{(|x^*|^2-r^2)^s}  \int_{\frac {|x^*|}{r}}^{\infty} \frac{(\tau^2-1)^{s-1} }{\tau^{n-1}} \, d\tau \\
 	=   \;&   k(n,s)\frac{\pi}{\sin(\pi s)} \frac{ |x^*|^{n-2}}{(|x^*|^2-r^2)^s}  \int_{\frac{|x^*|^2-r^2}{r^2}}^{\infty} \frac{t^{s-1}} {(t+1)^{n/2}} \, dt.
		\end{split}
	\end{equation*}
{	Using formula \eqref{firsttr} and definition \eqref{ro} we have the equalities
	\[\frac{|x^*|^2-r^2}{r^2}  =\frac{(r^2-|x|^2)(r^2-|z|^2)}  {r^2|x-z|^2}= r_0(x,z). \]
Therefore inserting $J(x^*)$ into \eqref{axzz} it follows that
	\[	A(x,z)=  k(n,s) |z-x|^{2s-n}  \int_{r_0(x,z)}^{\infty} \frac{t^{s-1}} {(t+1)^{n/2}} \, dt.\]}
By inserting this into \eqref{axz} we obtain that
\begin{equation*} G(x,z)=  a(n,s)|z-x|^{2s-n} \bigg(  1 - k(n,s)\int_{r_0(x,z)}^{\infty} \frac{t^{s-1}} {(t+1)^{n/2}} \, dt   \bigg).
\end{equation*}
 Now we change the variable $t=1/\tau^2 -1$ in definition \eqref{intkns} and obtain that
	\begin{equation}  k(n,s) \int_0^{\infty}  \frac{t^{s-1}} {(t+1)^\frac{n}{2}} \, dt =1.\label{knsinfty}\end{equation}
It follows that
	\begin{equation*}
		\begin{split}
	 G(x,z)
	=  a(n,s)k(n,s) |z-x|^{2s-n} \int_0^{r_0(x,z)}  \frac{t^{s-1}} {(t+1)^\frac{n}{2}} \, dt  .
		\end{split}
	\end{equation*}
We set
 \begin{equation}\kappa(n,s):=a(n,s)k(n,s)\label{kappa}\end{equation}
and conclude that
	\begin{equation*} G(x,z)= \kappa(n,s)|z-x|^{2s-n}  \int_0^{r_0 (x,z)}  \frac{t^{s-1}} {(t+1)^\frac{n}{2}} \, dt .    \end{equation*}  	
Hence the desired result in the case $n>2s$.
\end{proof}

\begin{proof}[Proof of Theorem \ref{theorem:thm1} for~$n<2s$]
We consider without loss of generality $r=1$ (by rescaling, the statement of the theorem is verified in the more general case). By \eqref{fundsolution} and definition \eqref{greendefn} we have that
\begin{equation}G(x,z) =  a(1,s) \big( |z-x|^{2s-1}-A(x,z)  \big), \label{axz11}\end{equation}
where \[A(x,z) := \int_{\mathbb{R}\setminus (-1,1)} \frac{P_1(y,x) }{|z-y|^{1-2s}} \, dy.\]
Using definition \eqref{poissondefn} we have that
\begin{equation*}
\begin{split}
A(x,z)= c(1,s)\int_{\mathbb{R}\setminus (-1,1)} \frac{ (1-x^2)^s} {|y-z|^{1-2s} (y^2-1)^s|y-x|} \, dy  .
		\end{split}
	\end{equation*}
We proceed exactly as in the case $n>2s$ performing the point inversion transformation and we arrive at
\begin{equation*}
		\begin{split}
		 A(x, z)=  c(1,s) |z-x|^{2s-1} \int_{-1}^1  \frac{  (x^{*^2}-1)^s }{  (1-y^{*^2})^s |x^*-y^*| } \, dy^*,
		\end{split}	
	\end{equation*}
	where $|x^*|>1$.
By symmetry we have that \begin{equation}
		 A(x, z)=   c(1,s) |z-x|^{2s-1} (x^{*^2}-1)^s |x^*|   J(x^*),\label{axzn1}
		\end{equation}
with \[J(x^*)=  \int_0^1  \frac{2 }{  (1-y^{*^2})^s (x^{*^2}-y^{*^2} )} \, dy^* .\]
We change the variable ${y^*}^2=t$ and obtain that
	\[J(x^*)= \Big(\frac{1}{x^*}\Big)^2 \int_0^1 t^{-1/2} (1-t)^{-s} \Big(1-\frac{t}{{x^*}^2} \Big)^{-1} \, dt.\]
By the integral representation \eqref{inthyp} of the hypergeometric function, it follows that
	\[ J(x^*)=   \Big(\frac{1}{x^*}\Big)^2 \, \frac{\Gamma(\frac{1}{2})\Gamma(1-s)}{\Gamma(\frac{3}{2}-s)} F\bigg(1, \frac{1}{2}, \frac{3}{2}-s, \frac{1}{{x^*}^2}\bigg).\]
We use the linear transformation \eqref{hyp4} (notice that $\big({1}/{x^*}\big)^2 <1$) and obtain that
	\begin{equation}
		\begin{split}
		 F\bigg(1, \frac{1}{2}, \frac{3}{2}-s, \frac{1}{{x^*}^2}\bigg)=\;& \frac{\Gamma(\frac{3}{2}-s) \Gamma(-s)}{\Gamma(\frac{1}{2}-s)\Gamma(1-s)} F\bigg(1,\frac{1}{2},s+1,\frac{{x^*}^2-1}{{x^*}^2}\bigg) \\
	+ \;&  \bigg(\frac{{x^*}^2\!-\!1}{{x^*}^2} \bigg)^{-s}  \frac{\Gamma(\frac{3}{2}\!-\!s) \Gamma(s)}{\Gamma(1)\Gamma(\frac{1}{2})} F\bigg(\frac{1}{2}\!-\!s, 1\!-\!s,1\!-\!s, \frac{{x^*}^2-1}{{x^*}^2}\bigg).
		\label{sumhyp}
		\end{split}
	\end{equation}
The first hypergeometric function obtained in the sum \eqref{sumhyp} is transformed according to \eqref{hyp3} as
	\[ F\bigg(1,\frac{1}{2},s+1,\frac{{x^*}^2-1}{{x^*}^2} \bigg)=  |x^*|F\bigg(\frac{1}{2},s,s+1,1-{x^*}^2 \bigg). \]
 For $s+1>s>0$ and $|1-{x^*}^2|<1 $ the convergence conditions are fulfilled for the integral representation \eqref{inthyp} of the hypergeometric function. Therefore we may write
	\[F\bigg(1,\frac{1}{2},s+1,\frac{{x^*}^2-1}{{x^*}^2} \bigg)= |x^*| \frac{\Gamma(s+1)}{\Gamma(s)} \int_0^1 \frac{t^{s-1}}{ \big(1+({x^*}^2-1) t\big)^{1/2}} \, dt .\]
On the other hand, for the second hypergeometric function obtained in identity \eqref{sumhyp}, we use transformations \eqref{hyp2} and \eqref{hyp3} and arrive at
	\begin{equation*}
		\begin{split}
			F\bigg(\frac{1}{2}-s, 1-s,1-s, \frac{{x^*}^2-1}{{x^*}^2}\bigg)= &\;{x^*}^{-2s} |x^*| F\bigg(\frac{1}{2}-s,0,1-s,1-{x^*}^2\bigg)\\
	=&\;{x^*}^{-2s} |x^*|  F\bigg(0,\frac{1}{2},1-s, \frac{{x^*}^2-1}{{x^*}^2}\bigg).
		\end{split}
	\end{equation*}
We use the Gauss expansion \eqref{gausshyp} with $a=0$, $b=  \frac{1}{2}$, $c=1-s$ and $w=  \frac{{x^*}^2-1}{{x^*}^2}$ (we notice that $0>c-a-b>-1$ for $s>1/2$ and $|w|<1$, thus the series is convergent). Since $a=0$, all the terms of the series vanish, except for $k=0$. Hence we obtain that
	\[F\bigg(0,\frac{1}{2},1-s, \frac{{x^*}^2-1}{{x^*}^2}\bigg)=1\]
and therefore
	\[F\bigg(\frac{1}{2}-s, 1-s,1-s, \frac{{x^*}^2-1}{{x^*}^2}\bigg)= {x^*}^{-2s} |x^*| .\]
Consequently
	\begin{equation*}
		\begin{split}
	J(x^*)=&\;\frac{1}{|x^*| } \bigg( \frac{ \Gamma(\frac{1}{2}) \Gamma(-s) \Gamma(s+1)}{\Gamma(\frac{1}{2}-s) \Gamma(s)}
	\int_0^1 \frac{t^{s-1}}{ \big(1+({x^*}^2-1) t\big)^{1/2}} \, dt
	+  \frac{ \Gamma(s)\Gamma(1-s)}{ ( {x^*}^2 -1)^{s}}\bigg).
		\end{split}
	\end{equation*}
{We recall that $c(1,s)= \big(\Gamma(s)\Gamma(1-s)\big)^{-1}$ and we define the constant
\begin{equation}k(1,s):= c(1,s) \frac{ \Gamma(\frac{1}{2}) \Gamma(-s) \Gamma(s+1)}{\Gamma(\frac{1}{2}-s) \Gamma(s)}.\label{k1s}\end{equation} We insert $J(x^*)$ into \eqref{axzn1} and have that
	\begin{equation*}
		\begin{split}
	   A(x,z)=  k(1,s) |z-x|^{2s-1}    \int_0^1 \frac{ ({x^*}^2-1)^s t^{s-1}}{ \big(1+({x^*}^2-1) t\big)^{1/2}} \, dt  +|z-x|^{2s-1} .			
		 \end{split}
	\end{equation*}
	With the change of variables $({x^*}^2 -1)t = \tau$ we obtain that
	\begin{equation*}
		\begin{split}
	   A(x,z)=  k(1,s) |z-x|^{2s-1}\int_0^{{x^*}^2-1} \frac{t^{s-1}} {(t+1)^{\frac{1}{2}}} \, dt  +|z-x|^{2s-1} .			
		 \end{split}
	\end{equation*}}
Inserting this into \eqref{axz11} and noticing that ${x^*}^2-1 =r_0(x,z)$ it follows that
		\[ G(x,z)= - a(1,s) k(1,s)  |z-x|^{2s-1}\int_0^{r_0(x,z)} \frac{t^{s-1}} {(t+1)^{\frac{1}{2}}} \, dt.\]
We call \begin{equation} \kappa(1,s) =- a(1,s) k(1,s) \label{kk1s}\end{equation} and conclude the proof of Theorem \ref{theorem:thm1} for $n<2s$.
\end{proof}

\begin{proof}[Proof of Theorem \ref{theorem:thm1} for~$n=2s$]
Without loss of generality, we assume $r=1$.  We insert the explicit formulas  \eqref{fundsolution} and \eqref{poissondefn} into definition \eqref{greendefn}. Moreover, we use the explicit values of the constant $  a\Big(1,\frac{1}{2}\Big)$ from \eqref{ctans3} and $ c\Big(1, \frac{1}{2}\Big)$ from \eqref{ctcns}. We obtain that
	 \begin{equation}\label{gxz111}
			G(x,z) =-\frac{1}\pi \log |x-z| + \frac{1}{\pi^2} \int_{|y|\geq 1} \log|y-z| \sqrt{\frac{1-x^2}{y^2-1}  } \frac{dy}{|x-y|}.
		\end{equation}
Let \[A(x,z):= \int_{|y|\geq 1} \log|y-z| \sqrt{\frac{1-x^2}{y^2-1}  } \frac{dy}{|x-y|}.\]
We perform the change of variables $v=  \frac{yx-1}{y-x}$. Since $1-v^2 \geq 0$, we have that $|v|\leq 1$. We set $w:=  \frac{xz-1}{z-x}$ and observe that $|w|\geq 1$.
It follows that
	\[A(x,z) = \int_{|v|\leq 1}   \bigg(\log\frac{ |v-w|}{|v-x|}+ \log|z-x|\bigg) \, \frac{dv}{\sqrt{1-v^2}}. \]
We use identity \eqref{logid} and since $|w|\geq 1$ and $|x|\leq 1$ we obtain that
 \begin{equation*}
	\begin{split}
	  A(x,z) =&\; \pi \log \Big(|w|+(w^2-1)^{1/2}\Big)  + \pi \log |x-z| \\
	= &\; \pi \log (1-zx + \sqrt{(1-x^2)(1-z^2)} ).
	\end{split}
\end{equation*}
Inserting this into \eqref{gxz111} we obtain that
	\[G(x,z)=\frac{1}\pi \log \bigg(\frac{1-zx + \sqrt{(1-x^2)(1-z^2)}}{|x-z|} \bigg).\]
This completes the proof of Theorem \ref{theorem:thm1} for~$n=2s$.
\end{proof}

\subsection{Representation formula for the fractional Poisson equation}\label{thm2}

This subsection is dedicated to the proof of Theorem~\ref{theorem:thm2}, which we recall here.

\begin{theorem} \label{theorem:thm2}
Let $r>0$, $h \in C^{2s+\eee}(B_r)\cap C( \overline B_r)$ and let
	 \begin{equation*}
		u(x) : =
			\begin{cases}
		\displaystyle \int_{B_r} h(y) G(x,y) \, dy \quad & \text{ if } x\in B_r, \\
		0 \quad \quad & \text{ if } x \in {\obal}.
		\end{cases}
	\end{equation*}
Then $u$ is the unique pointwise continuous solution of the problem \eqref{PoissoneqD}
	\begin{equation*}
		\begin{cases}
		    \frlap u= h   \qquad &\mbox{ in } {B_r} ,
 		\\  u= 0   \qquad &\mbox{ in } {\obal}.
	\end{cases}
	\end{equation*}
\end{theorem}	

\begin{proof}[Proof of Theorem \ref{theorem:thm2}]
We identify $h$ with its $C_c^{2s+\eee} (\Rn)$ extension, namely we consider $\tilde h \in C_c^{2s+\eee}(\Rn)$ with $B_r \subset \text{supp} \,\tilde h$ such that $\tilde h=h$ on $B_r$.
Then, by definition \eqref{greendefn} we have that in $B_r$
	\begin{equation*}
		\begin{split}
		 u(x) = &\; \int_{B_r}h(z) G(x,z)\, dz \\
		 	= &\; \int_{B_r} h(z)\Phi(z-x)  dz -   \int_{B_r} h(z) \bigg(\int_{\obal} \Phi(y-z)P_r(y,x) dy \bigg)  dz \\
	=&\;  h*\Phi(x) - \int_{\obal} P_r(y,x) \big(h*\Phi\big)(y) \, dy.
		\end{split}
	\end{equation*}
Let  \[ g(x):= h*\Phi(x) \mbox{ for any  } x \in \Rn.\] As we have seen in Theorem~\ref{theorem:poissonsolution},~$g \in L_s^1(\Rn) \cap C(\Rn)$.
Let for any $x \in \Rn$ \[ u(x)=v_0(x)-v_1(x),\] where
$v_0(x) = g(x) \mbox{ in } \Rn$
and
\[ v_1(x)=	
		\begin{cases}
	\displaystyle  \int_{{\Rn}\setminus B_r} P_r(y,x) g(y)\, dy &\quad  \, \text{if } x\in B_r, \\
				g(x) &\quad \, \text{if } x \in {\obal}.
		\end{cases}
\]
Then for $x\in B_r$, thanks to Theorems \ref{theorem:poissonsolution} and \ref{theorem:DPL}
	\[ \frlap u(x) = h(x) - 0=h(x),\]
hence $u$ is solution \eqref{PoissoneqD}. Also, Theorem \ref{theorem:poissonsolution} and Theorem \ref{theorem:DPL} assure the continuity of $u$ in $\Rn$.

The uniqueness of the solution follows from the simple application of the Maximum Principle for the fractional Laplacian (see Theorem 3.3.3 in \cite{nonlocal}).
\end{proof}

\subsection{Computation of the normalization constants}\label{constants}

This subsection is dedicated to the computation of the constant $\kappa(n,s)$ in Theorem \ref{theorem:thm1}.

We first compute the constant $k(n,s)$ given in \eqref{intkns} for $n>2s$.
We claim that
 \begin{equation} k(n,s)=  \frac{\Gamma({\frac{n}{2}})} { \Gamma(\frac{n}{2}-s) \Gamma(s)}.\label{kns} \end{equation}

Indeed, using definition \eqref{intkns} and taking the change of variable $\tau^2=t$ we have that
\[ \frac{1}{k(n,s)} = 2 \int_0^1  \tau^{n-2s-1} (1-\tau^2)^{s-1} \, dt =  \int_0^1 t^{\frac{n}{2}-s -1} (1-t)^{s-1} \, dt .\]
We use identities \eqref{betazerouno} and \eqref{betagamma} to obtain that
\[ \int_0^1  t^{\frac{n}{2}-s-1} (1-t)^{s-1} \, dt =  \frac{\Gamma( \frac{n}{2}  -s)\Gamma(s)}{\Gamma(\frac{n}{2} )},\]
which is exactly the result.

We now prove Theorem \ref{theorem:kns}, namely we compute the constant $\kappa(n,s)$ encountered in the formula of the Green function $G$.
\begin{theorem} \label{theorem:kns}The constant $\kappa(n,s)$ introduced in identity \eqref{forgkns} is
	\begin{align*}
		 \kappa(n,s) &= \displaystyle \frac{\Gamma(\frac{n}{2}) } {2^{2s}\pi^{ \frac{n}2}   \Gamma^2(s) } & \text{ for } &n \neq 2s, \\
		\kappa\Big(1,\frac{1}{2}\Big) &=\frac{1}{\pi}
		  & \text{ for } &n=2s.
		\end{align*}
\end{theorem}
\begin{proof}[Proof of Theorem \ref{theorem:kns}]
For $n>2s$, we insert the values of $a(n,s)$ from \eqref{ctans1} and of $k(n,s)$ from \eqref{kns} into  \eqref{kappa} and we obtain that
\begin{equation*} \begin{aligned}
\kappa(n,s) =a(n,s)k(n,s)
 =  \frac{ \Gamma(\frac{n}{2}) }{2^{2s}\pi^{\frac{n}2} \Gamma^2(s)}.
\end{aligned} \end{equation*}
For $n<2s$, we recall definitions \eqref{k1s}, \eqref{kk1s} and \eqref{ctans1}, we use  identities \eqref{gam3}, \eqref{gam4} and \eqref{gamxx1} relative to the Gamma function and obtain that
	\[ \begin{aligned}
	\kappa(1,s)=   -a(1,s) k(1,s)
		=  \frac{(-s)\Gamma(-s) }{2^{2s}\Gamma(s)} \frac{1}{\Gamma(1-s)\Gamma(s)}
		=\frac{1}{2^{2s}\Gamma^2(s)}. \end{aligned} \]
On the other hand, we recall that $\kappa\Big( 1,\frac{1}{2}\Big)=  \frac{1}\pi$, as we have seen in the proof of Theorem \ref{theorem:thm1}  for $n=2s$.
This concludes the proof of Theorem \ref{theorem:kns}.
\end{proof}

We make now a remark on the constants $C(n,s)$ defined by \eqref{GCNS} and $c(n,s)$ defined in \eqref{ctcns}. These two constants are both used in various works in the definition of the fractional Laplacian, but as we have seen in the course of this paper, they arise for different normalization purposes. The constant $C(n,s)$ as defined by \cite{galattica} is consistent with the Fourier expression of the fractional Laplacian, meanwhile $c(n,s)$ as introduced in \cite{Landkof} is used to normalize the Poisson kernel (and the $s$-mean kernel), and is consistent with the constants used for the fundamental solution and the Green function. We first introduce the direct computation of the fractional Laplacian of a particular function. Namely:

\begin{lemma}
Let $u(x) =(1-|x|^2)_{+}^s$. Then in $B_1$
	\[ \frlap u(x) = C(n,s)\frac{\omega_n}{2} B(s, 1-s),\]
where $B$ is the Beta function defined in \eqref{beta}.
\label{dydares}
\end{lemma}
{The more general case (more precisely, for the function $u(x) =(1-|x|^2)_{+}^p$ for any $p>-1$)  was proved in \cite{Dyda,Dyda2}. With small modifications with respect to the general case, the proof of Lemma \ref{dydares} can be also found in Section 3.6 in \cite{nonlocal}.}

\begin{theorem}\label{thm:Cc}
The constant $C(n,s)$ introduced in \eqref{GCNS} is given by
	\begin{equation} \label{cnscomputed} C(n,s) = \frac{2^{2s} s \Gamma\left(\frac{n}2 +s\right)} {\pi^{\frac{n}2} \Gamma(1-s)}.\end{equation}
\end{theorem}

\begin{proof}[Proof of Theorem \ref{thm:Cc}]
By Lemma \ref{dydares} we have that in $B_1$
 \[\frlap u(x) = C(n,s) \frac{\omega_n}{2} B(1-s,s).\]
We use Theorem \ref{theorem:thm2} and for $n\neq 2s$, we obtain that
	\begin{equation*}
		\begin{split}
		u(x) = &\; \int_{B_1} C(n,s) \frac{\omega_n}{2}  B(1-s,s) G(x,y) dy \\
	= &\;  C(n,s) \frac{\omega_n}{2} B(1-s,s)  \kappa(n,s) \int_{B_1} |x-y|^{2s-n} \bigg( \int_0^{r_0(x,y)} \frac {t^{s-1}}{(t+1)^{\frac{n}{2}}}\, dt\bigg) \, dy.
		\end{split}
	\end{equation*}
We compute the latter identity in zero and have that
 \begin{equation}
		\begin{split}
		&1= C(n,s) \frac{\omega_n}{2} B(1-s,s)  \kappa(n,s) \int_{B_1} |y|^{2s-n} \bigg( \int_0^{\frac{1-|y|^2}{|y|^2}} \frac {t^{s-1}}{(t+1)^{\frac{n}{2}}}\, dt\bigg) \, dy .\label{constant}
		\end{split}
	\end{equation}
We compute the double integral, by using Fubini-Tonelli's theorem
 \begin{equation*}
		\begin{split}
		\int_{B_1} |y|^{2s-n} \bigg( \int_0^{\frac{1-|y|^2}{|y|^2}} \frac {t^{s-1}}{(t+1)^{\frac{n}{2}}}\, dt\bigg) \, dy  =&\; \omega_n   \int_0^1 \rho^{2s-1} \bigg( \int_0^{\frac{1-\rho^2}{\rho^2}} \frac {t^{s-1}}{(t+1)^{\frac{n}{2}}}\, dt\bigg) \, d\rho\\
	 = &\; \omega_n \int_0^{\infty} \frac {t^{s-1}}{(t+1)^{\frac{n}{2}}} \bigg( \int_0^{\frac{1}{\sqrt{t+1}}} \rho^{2s-1} \, d \rho\bigg) \, dt \\=&\; \frac{\omega_n}{2s} \int_0^{\infty} \frac {t^{s-1}}{(t+1)^{\frac{n}{2}+s}}\, dt =   \frac{\omega_n}{2s} B\bigg(s,\frac{n}{2}\bigg).
		\end{split}
	\end{equation*}
By inserting this, the value of $\kappa(n,s)$ from Theorem \ref{theorem:kns} and the measure of the $(n-1)$-dimensional unit sphere $\omega_n=({ 2\pi^{n/2}})/{\Gamma(n/2)}$ into \eqref{constant} and using \eqref{betagamma} we obtain that
\[ C(n,s) = \frac{2^{2s} s \Gamma(\frac{n}2 +s)} {\pi^{\frac{n}2 }\Gamma(1-s)}.\]
For $n=2s$ we have 	that $\frlap u(x) = C\left(1, 1/2\right) \pi.$
Thanks to Theorem \ref{theorem:thm2}
	\[ u(x) =   C\left(1, \frac{1}2\right) \pi \int_{-1}^1 G(x,y)\, dy.\]
	Using formula \eqref{formn1s12} and computing $u$ at zero, we obtain that
		\[ \begin{aligned} 1 =  C\left(1, \frac{1}2\right)  \int_{-1}^1 \log \frac{ 1+\sqrt{1-y^2}}{|y|}\, dy
			=  {\pi} C\left(1, \frac{1}2\right) .\end{aligned} \]
			Hence $ C\left(1, 1/2\right)  = {1}/{\pi}$ and this concludes the proof of the Theorem.
\end{proof}

\appendix

\section{Appendix} \label{appendix}
\subsection{The Gamma, Beta and hypergeometric functions}
We recall here a few notions on the special functions Gamma, Beta and hypergeometric (see \cite{gamma}, Chapters 6 and 15 for details).\\

\textbf{Gamma function.} The Gamma function is defined  for $x>0$ as (see \cite{gamma}, Chapter 6):
\begin{equation} \Gamma(x):=\int_0^{\infty}t^{x-1}e^{-t}\, dt.\label{gamma}\end{equation}
The Gamma function has an unique continuation to the whole $\mathbb{R}$ except at the negative integers, by means of Euler's infinite product.
We have that $\Gamma(1)=\Gamma(2)=1$ and $\Gamma({1}/{2})= \sqrt{\pi}$. We  also recall the next useful identities:
 \begin{align}
			&\Gamma(n+1)= n! &   \mbox{ for any } &n \in \N,\label{gammafac}\\
			&\Gamma(x+1)=x\Gamma(x) &   \mbox{ for any } &x >0,   \label{gamxx1}\\
			&\frac{\Gamma(1/2+x)}{\Gamma(2x)} =  \frac{\sqrt{\pi} 2^{1-2x} }{\Gamma(x)}&   \mbox{ for any } &x >0 , \label{gam2}\\
			 	&\Gamma(s)\Gamma(1-s) \;\;=  \frac{\pi}{\sin (\pi s)}&   \mbox{ for } &s\in (0,1),   \label{gam3}\\
			& \Gamma(1/2-s) \Gamma(1/2+s) =   \frac{\pi}{\cos(\pi s)}&   \mbox{ for } &s\in (0,1),   \label{gam1}\\
	 	 &  \Gamma(1-s) =  (-s)\Gamma(-s) &   \mbox{ for } &s\in (0,1)\label{gam4}.
	 \end{align}

\textbf{Beta function.}
The Beta function can be represented as an integral (see \cite{gamma}, Section 6.2), namely for $x, y>0$
	 \begin{equation}
			 B(x,y) =  \int_0^{\infty} \frac {t^{x-1}}{(1+t)^{x+y}} \, dt \label{beta}
	\end{equation}
and equivalently
	\begin{equation}
 	B(x,y)=  \int_0^1 t^{x-1}(1-t)^{y-1} \, dt \label{betazerouno}.
	\end{equation}
Furthermore, in relation to the Gamma function we have the identity
	\begin{equation}
B(x,y)= \frac {\Gamma(x)\Gamma(y)} {\Gamma(x+y)} \label{betagamma}.
	\end{equation}

\textbf{Hypergeometric functions.} There are several representations for the hypergeometric function (see \cite{gamma}, Chapter 15). We recall the ones useful for our own purposes. \\

\emph{(1)  Gauss series}
		\begin{equation}  F(a,b,c,w) = \sum_{k=0}^{\infty} \frac{ (a)_k (b)_k }{(c)_k} \frac{w^k} { k!},\label{gausshyp} \end{equation}
	where $(q)_k$ is the Pochhammer symbol defined by:
		\begin{equation}\label{Pochh}
		(q)_k = \begin{cases}
			1   &\mbox{ for } k = 0 ,\\
  		q(q+1) \cdots (q+k-1) &\mbox{ for } k > 0.
		 \end{cases}
		\end{equation} \\
The interval of convergence of the series is $|w|\leq 1$. The Gauss series, on its interval of convergence, diverges when $c-a-b\leq -1$, is absolutely convergent when $c-a-b>0$ and is conditionally convergent when $|w|<1$ and $-1<c-a-b\leq 0$.	
Also, the series is not defined when $c$ is a negative integer $-m$, provided $a$ or $b$ is a positive integer $n$ and $n<m$.

Some useful elementary computations are
	\begin{subequations}	
		 \begin{align}
		&F(a,b,b,w)=(1-w)^{-a}.\label{hypelc1}\\
		&F\Big(a, \frac{1}{2}+a, \frac{1}{2},w^2\Big) = \frac{ (1+w)^{-2a} +(1-w)^{-2a}}{2}.\label{hypelc2}
		\end{align}
	\end{subequations}

\emph{(2)  Integral representation}
	\begin{equation}
	 F(a,b,c,w) :=  \frac{\Gamma(c)}{\Gamma(b) \Gamma(c-b)} \int_0^1 t^{b -1} (1-t)^{c-b-1} (1- w t)^{-a}  \, dt .\label{inthyp}
	\end{equation}
The integral is convergent (thus $F$ is defined) when $c>b>0$ and $|w|<1$. \\

\emph{ (3) Linear transformation formulas}

From the integral representation \eqref{inthyp}, the following transformations can be deduced.
	\begin{subequations}
		 \begin{align}	
		F(a,b,c,w) =&\;  (1-w)^{c-a-b} F(c-a,c-b,c,w), \label{hyp1}\\
			 =&\;(1-w)^{-a} F\Big(a,c-b,c, \frac{w}{w-1}\Big),\label{hyp2} \\
				 =&\;(1-w)^{-b} F\Big(b,c-a,c,\frac{w}{w-1}\Big),\label{hyp3}\\
				= &\;\frac{\Gamma(c)\Gamma(c-a-b)}{\Gamma(c-a)\Gamma(c-b)} F(a,b,a+b-c+1,1-w) \notag \\&\;+ (1-w)^{c-a-b} \frac{\Gamma(c)\Gamma(a+b-c)}{\Gamma(a)\Gamma(b)}F(c-a,c-b,c-a-b+1, 1-w), \notag\\ &\text{ when } 0<w<1\label{hyp4} .
		\end{align}
	\end{subequations}

\subsection{Point inversion transformations}
The purpose of this appendix is to recall some basic geometric features of the point inversion, related to the so-called Kelvin transformation.

Let $r>0$ to be fixed.

\begin{definition}
Let $x_0\in B_r$ be a fixed point. The inversion with center $x_0$ is a point transformation that maps an arbitrary point $y \in {\Rn}\setminus \{x_0\}$ to the point $ \mathbf{K}_{x_0}(y)$ such that the points $y$, $x_0$, $ \mathbf{K}_{x_0}(y)$ lie on one line, $x_0$ separates $y$ and $\mathbf{K}_{x_0}(y)$ and
	\begin{equation}
	\mathbf{K}_{x_0}(y):=x_0- \frac{r^2-|x_0|^2}{|y-x_0|^2} \big( y-x_0\big). \label{trk}
	\end{equation}
\end{definition}
This is a bijective map from ${\Rn}\setminus \{ x_0\} $ onto itself. Of course, $\mathbf{K}_{x_0}\big( \mathbf{K}_{x_0}(x) \big) = x$. When this does not generate any confusion, we will use the notation $y^*:=\mathbf{K}_{x_0}(y)$ and $x^*:=\mathbf{K}_{x_0}(x)$ to denote the inversion of $y$ and $x$ respectively, with center at $x_0$.

\begin{remark}
It is not hard to see, from definition \eqref{trk}, that
	 \begin{equation}
	 |y^*-x_0||y-x_0|=r^2-|x_0|^2. \label{tr}
	\end{equation}

\end{remark}

\begin{proposition}
\label{proposition:pointinv}
Let $x_0 \in B_r$ be a fixed point, and $x^*$ and $y^*$ be the inversion of $x \in {\Rn}\setminus \{x_0\}$ respectively $y \in {\Rn}\setminus \{x_0\}$ with center at $x_0$.
Then:
	\begin{subequations}
		\begin{align}
			\text{ a) }  & \text{points on the sphere } \partial B_r \text{ are mapped into points on the same sphere},  \quad \quad \quad\quad \quad \quad \quad \quad \quad \notag 
			\\
			 \text{ b) } & \text{points outside the sphere } \partial B_r  \text{ are mapped into points inside the sphere}, \notag 
			 \\
	          		\text{ c) } &\frac{|y-x_0|^2}{(r^2-|x_0|^2)(r^2-|y|^2)}=\frac{1}{|y^*|^2-r^2},  \label{firsttr}	\\
			\text{ d) } & \frac{dy}{|y-x_0|^n}=\frac{dy^*}{|y^*-x_0|^n},  \label{dxtr} \\
			\text{ e) } & |y^*-x^*|=(r^2-|x_0|^2)\frac{|y-x|}{|y-x_0||x-x_0|}.  \label{sectr}	
		\end{align}
	\end{subequations}
\end{proposition}

{The Kelvin point inversion transformation is well known (see, for instance, the Appendix in \cite{Landkof}) and elementary geometrical considerations can be used to prove this lemma. We give here a sketch of the proof. }
\begin{proof}[Sketch of the proof]
A simple way to prove claims a) to c) is to consider the first triangle in Figure \ref{fign:figure1}.

\begin{figure}[htpb]
	\hspace{0.85cm}
	\begin{minipage}[b]{0.98\linewidth}
	\centering
	\includegraphics[width=0.98\textwidth]{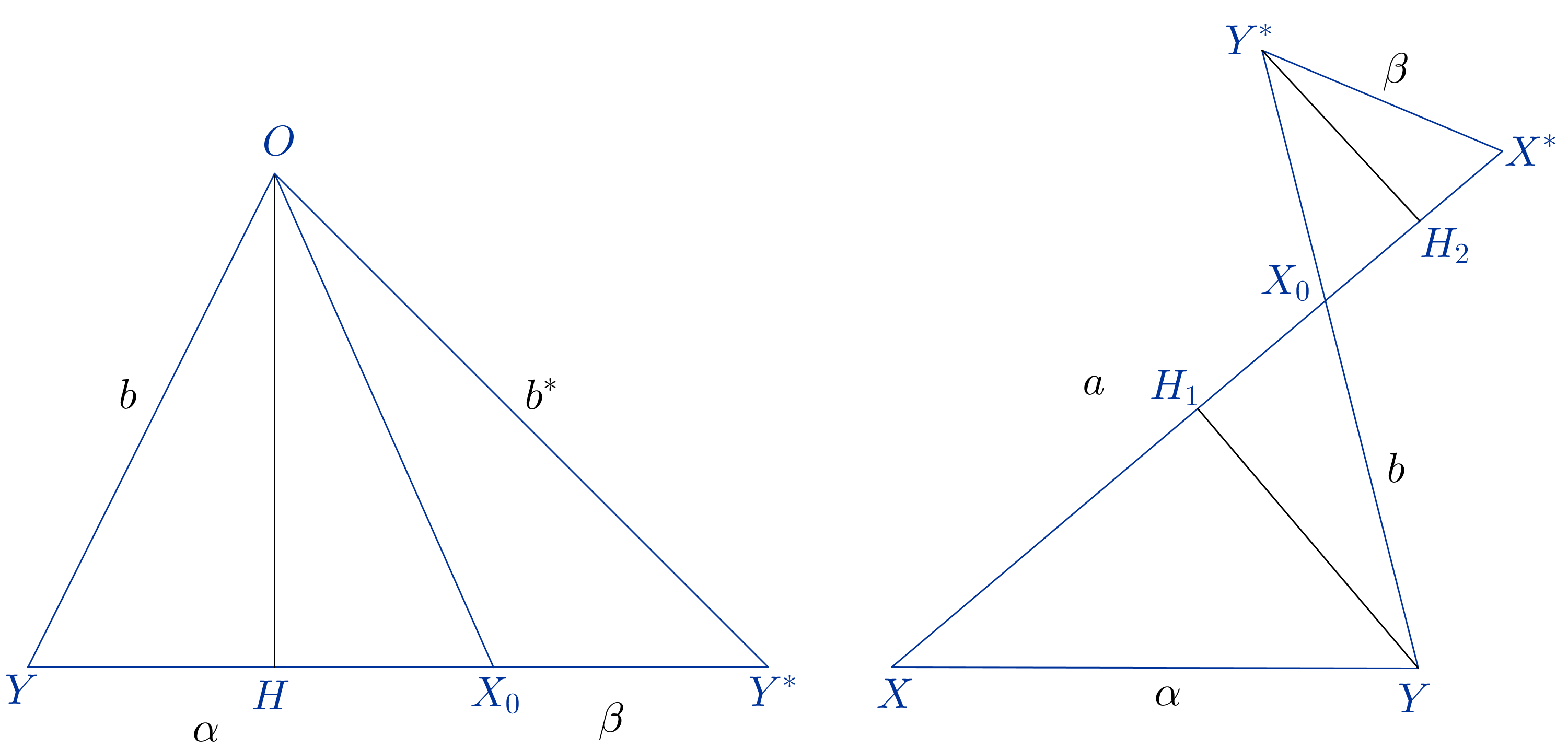}
	\caption{Inversion of $x, y$ with center at $x_0$}
	\label{fign:figure1}
	\end{minipage}
	\end{figure}
%
%
We denote
 $b:=|OY|=|y|$, $b^*:=|OY^*|=|y^*|$, $\alpha:=|X_0Y|=|y-x_0|$ and $\beta := |X_0Y^*|=|y^*-x_0|$.
 Let $OH$ be the perpendicular from $O$ onto $YY^*$.
 We apply the Pythagorean Theorem in the three triangles $\triangle OYH, \, \triangle OHX_0, \, \triangle OHY^*$,  add  the equation \eqref{tr} and by solving the system, one gets that
	 \begin{equation*}
		 {b^*}^2= \frac{\beta r^2+\alpha r^2-\beta b^2}{\alpha}. \label{sol}
	\end{equation*} From this, claims a) to c) follow after elementary computations.
	
In order to prove d), without loss of generality, one can consider the point inversion of radius one with center at zero
	$y^*={-y}/{|y|^2} $ and take its derivative. Since the point inversion transformation is invariant under rotation, we can assume that $y=|y|e_1$ and the desired result plainly follows.
%

To prove e), see the Appendix in \cite{Landkof}, or consider the second triangle in Figure \ref{fign:figure1}. We denote
$a :=|X_0X|=|x-x_0|$,
$b:= |X_0Y|=|y-x_0|$,
$\alpha :=|XY|=|x-y|$ and $\beta:=|X^*Y^*|=|x^*-y^*|$. Let $YH_1$ and $Y^*H_2$ be perpendiculars from $Y$, respectively $Y^*$ onto the segment $XX^*$.
 By applying the Pythagorean Theorem in the four triangles $\triangle X_0YH_1,\, \triangle XYH_1, \,  \triangle X_0Y^*H_2$, $\triangle X^*Y^*H_2$, adding relation \eqref{tr}  and using that $YH_1$ is parallel to $Y^*H_2$, one gets after solving the system that
\begin{equation*}  \beta=\frac{(r^2-|x_0|^2) \alpha}{ab}, \end{equation*}	which is the desired result.
\end{proof}

\subsection{Some useful integral identities}

We present here a few detailed computations related to the functions $\Phi$, $A_r$ and $P_r$ and some other useful integral identities.

\begin{lemma} For any $r>0$
\begin{equation}  \int _{\obal} A_r(y) \, dy =1. \label{Ir} \end{equation}
\end{lemma}

\begin{proof}
Using \eqref{smeandefn} and passing to polar coordinates we have that
 	\begin{equation*}
 		\begin{split}
 		 \int _{\obal} A_r(y) \, dy= &\;  c(n,s) \int _{\obal}   \frac {r^{2s} } {(|y|^2 -r^2)^s|y|^n}\, dy\\
 			= &\; c(n,s)  \omega_n  \int_r ^{\infty}   \frac { r^{2s} }{\rho  (\rho^2-r^2)^s} \, d\rho,
 		\end{split}
 	\end{equation*}
where $\omega_n={ 2\pi^{n/2}}/{\Gamma(n/2)}$  is the measure of the $(n-1)$-dimensional unit sphere.
We change the variable $z= (\rho/r)^2-1$ and have that
	\begin{equation}
\int _{\obal} A_r(y) \, dy
  = \frac {c(n,s)}{2} 	 \omega_n  \int_0 ^{\infty}  \frac  {1} {(z+1) z^s} \, dz .  \label{Ircalc}   	 \end{equation}
We apply identities \eqref{beta} and \eqref{betagamma} and use identity \eqref{gam3} to obtain that
	\begin{equation} \int_0^{\infty}    \frac  {1} {(z+1) z^s} \, dz  = \Gamma (1-s) \Gamma(s) =  \frac{\pi}{\sin{\pi s} }.\label{betappl} \end{equation}
Using the definition \eqref{ctcns} of $c(n,s)$ it follows that $ \int _{\obal} A_r(y) \, dy= 1$, as desired.
\end{proof}

\begin{lemma} \label{lemmaip}
For any $r>0$ and any $x \in B_r$
\begin{equation} \int_{\obal}P_r(y,x) \, dy  =1. \label{Ip} \end{equation}
\end{lemma}

\begin{proof}
From the definition \eqref{poissondefn} of $P_r$ we have that \[\int_{\obal}P_r(y,x) \, dy= c(n,s) \int_{\obal} \Bigg (\frac {r^2-|x|^2}{|y|^2-r^2}\Bigg)^s \frac {1}{|x-y|^n}  \, dy .\]

We make the proof for $n>3$. However, the results hold for $n\leq 3$. We change variables using the hyperspherical coordinates with radius $\rho>0$ and angles $\theta, \theta_1, \dots, \theta_{n-3} \in [0, \pi], \theta_{n-2} \in [0,2\pi]$
	\begin{equation}
		\begin{split} \label{chofvarhypersph}
		 y_1= &\rho \sin\theta\sin\theta_1 \dots \sin\theta_{n-3} \sin \theta_{n-2}\\
 	             y_2=  &\rho\sin \theta \sin \theta_1 \dots \sin \theta_{n-3} \cos \theta_{n-2} \\
   	             y_3= &\rho \sin \theta \sin \theta_1 \dots \cos \theta_{n-3} \\
   	       \dots \\
		 y_n=&\rho \cos\theta.
		\end{split}
	\end{equation}
	The Jacobian of the transformation is given by $ \rho^{n-1} \sin ^{n-2}\theta  \sin ^{n-3}\theta_1 \dots \sin \theta_{n-3} $.
We only remark that for $n=3$ the usual spherical coordinates can be used $	y_1= \rho \sin\theta\sin\theta_1,  y_2=  \rho\sin \theta \cos \theta_1  \mbox{ and } y_3= \rho \cos \theta $, while for $n=2$ and $n=1$ similar computations can be performed.

Without loss of generality and up to rotations, we assume that $x =|x|e_n$ to obtain the identity $ |x-y|^2 = \rho^2+|x|^2-2|x| \rho\cos\theta$ (see Figure~\ref{fign:xycalc} for clarity). With this change of coordinates, we obtain
%
%
	\begin{equation*}
		\begin{split}
		& \int_{\obal}P_r(y,x) \, dy\\
		=& c(n,s) (r^2  - |x|^2)^s 2\pi \prod_{k=1}^{n-3}\int_0^{\pi} \sin^k \theta \,d\theta
   \int_r^{\infty}\int_0^{\pi}  \frac{\rho^{n-1}\sin^{n-2} \theta \, d\theta \, d\rho }{(\rho^2-r^2)^s (\rho^2+|x|^2 - 2\rho |x| \cos \theta)^{n/2}} .
		\end{split}
	\end{equation*}
	\begin{figure}[htb]
	\hspace{0.85cm}
	\begin{minipage}[b]{0.78\linewidth}
	\centering
	\includegraphics[width=0.78\textwidth]{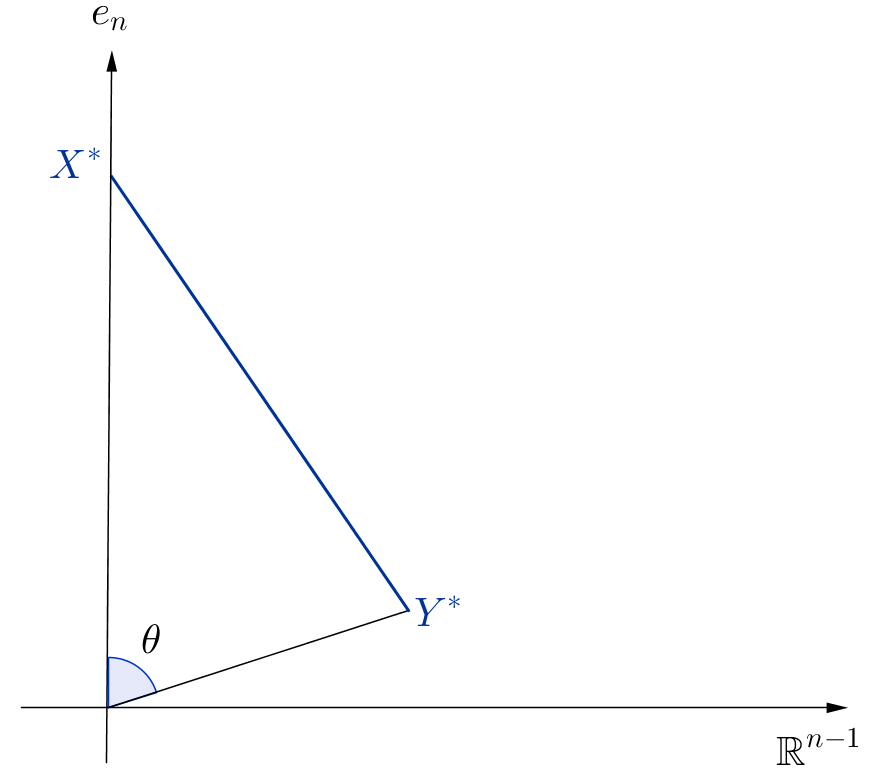}
	\caption{}
	\label{fign:xycalc}
	\end{minipage}
	\end{figure}
We do the substitution $\bar r= {r}/{|x|} $ and $\bar \rho= {\rho}/{ |x|}$ but still use $r$ and $\rho$ for simplicity and we remark that now $\rho>1$ and $r>1$.
 We obtain that
	\begin{equation}
\label{ipxxx}
		\begin{split}
	&  \int_{\obal}P_r(y,x) \, dy\\
		 =&\; c(n,s) (r^2-1)^s 2\pi  \prod_{k=1}^{n-3}\int_0^{\pi}  \sin^k\theta\,  d\theta
\int_r^{\infty} \frac {\rho^{n-1}}{(\rho^2-r^2)^s}\bigg( \int_0^{\pi} \frac  {\sin^{n-2}\theta \,d\theta} {(\rho^2+1-2\rho \cos\theta )^{n/2}}  \bigg)\, d\rho.
		\end{split}
	\end{equation}
Let
	\begin{equation*}
	 i (\rho):=\int_0^{\pi} \frac {\sin^{n-2}\theta}{(\rho^2-2\rho \cos\theta +1)^{n/2}} \, d\theta.
	\end{equation*}
We claim that, for $\rho >1$
\begin{equation} \label{aahhh}
	 i(\rho)=\frac {1} {\rho^{n-2} (\rho^2-1)}\int_0^{\pi} \sin^{n-2} \theta \, d\theta.
	\end{equation}
To prove this, we use the following change of coordinates
\begin{equation}  \frac {\sin\theta}{\sqrt {\rho^2 -2\rho \cos\theta +1 }} = \frac {\sin \alpha }{\rho}. \label{changeofvar} \end{equation}
We have that
	\begin{equation} d\theta = \Bigg( 1- \frac{\cos \alpha}{\sqrt{\rho^2-\sin^2\alpha}}\Bigg) \, d\alpha. \label{chvar} \end{equation}
{ To see this, one takes the derivative of the relation \eqref{changeofvar}
	\begin{equation}
		 \frac{(\rho \cos \theta -1)(\rho - \cos \theta)} {(\rho^2- 2\rho \cos \theta +1)^{\frac{3}{2}}}\, d\theta = \frac{\cos \alpha}{\rho}\, d \alpha \label{capduedue}
	\end{equation}
and obtains with some manipulations of \eqref{changeofvar} that}
\[  \frac{ \cos \alpha  \sqrt{\rho^2 - \sin^2\alpha} }{ \sqrt{\rho^2 - \sin^2\alpha} -\cos \alpha} =\frac{\rho(\rho \cos \theta -1)(\rho - \cos \theta)} {(\rho^2- 2\rho \cos \theta +1)^{\frac{3}{2}}} .\]
Now by changing variables we obtain that
	\begin{equation*}
		\begin{split}
		 i(\rho)= &\;\int_0^{\pi} \frac {\sin^{n-2}\theta}{(\rho^2-2\rho \cos\theta +1)^{n/2}} \, d\theta  \\
		=&\; \frac{1}{\rho^{n-2}} \int_0^{\pi}  \frac{\sin ^{n-2}\alpha \, d\alpha}{(\sqrt{\rho^2- \sin^2 \alpha} -\cos \alpha ) \sqrt{\rho^2- \sin^2 \alpha}} \\
= &\; \frac{1}{\rho^{n-2}} \int_0^{\pi}   \frac{\sin ^{n-2}\alpha (\sqrt{\rho^2- \sin^2 \alpha} +\cos \alpha ) \, d\alpha}{ (\rho^2-1) \sqrt{\rho^2- \sin^2 \alpha}} \\
= &\;\frac{1}{\rho^{n-2}(\rho^2-1) } \bigg( \int_0^{\pi} \sin ^{n-2} \alpha  \, d\alpha + \int_0^{\pi} \frac{\sin^{n-2} \alpha \, \cos \alpha}{   \sqrt{\rho^2- \sin^2 \alpha}}\, d\alpha\bigg).
		\end{split}
	\end{equation*}
By symmetry \[\int_0^{\pi} \frac{\sin^{n-2} \alpha \, \cos \alpha}{   \sqrt{\rho^2- \sin^2 \alpha}}\, d\alpha  =0,\]
therefore \[ i(\rho)= \frac{1}{\rho^{n-2}(\rho^2-1) }  \int_0^{\pi} \sin ^{n-2} \alpha  \, d\alpha .\]
We substitute this into \eqref{ipxxx} and obtain that
	\[\int_{\obal}P_r(y,x) \, dy=c(n,s) (r^2-1)^s 2\pi \prod_{k=1}^{n-2} \int_0^{\pi} \sin^k\theta\,  d\theta  \int_r^{\infty} \frac{\rho\,  d\rho}{  (\rho^2-r^2)^s (\rho^2-1)}.\]
We claim that
	\begin{equation}
	\pi \prod_{k=1}^{n-2} \int_0^{\pi}\sin^k\theta d\theta = \frac{{\pi}^{n/2}}{\Gamma (n/2)}.
	\label{sinstuff}
	\end{equation}
To prove this, we integrate by parts and obtain that
	\begin{equation*}
		\begin{split}
		 I_k=  \int_0^{\pi} \sin^k\theta \, d\theta
		=(k-1) \int_0^{\pi} \sin^{k-2}\theta \, d\theta - (k-1)\int_0^{\pi}  \sin^k \theta \, d \theta  ,
		\end{split}
	\end{equation*}
which implies that \[  I_k=  \frac{k-1}{k}\int_0^{\pi}  \sin^{k-2}\theta \, d\theta = \frac{k-1}{k} I_{k-2}.  \]
Thus we have
	\begin{equation*}
                         I_k=
		\displaystyle \begin{cases}
			\displaystyle \frac{k-1}{k}\frac{k-3}{k-2}\dots \frac{1}{2}I_0 \quad   \text{  if $k$ even}, \\
			\displaystyle  \frac{k-1}{k}\frac{k-3}{k-2}\dots \frac{2}{3}I_1\quad  \text{  if $k$ odd},
		 \end{cases}
	\end{equation*}
with $I_0=\pi$ and $I_1=2$, and the claim \eqref{sinstuff} follows after elementary computations. And so
	\[\int_{\obal}P_r(y,x) \, dy=c(n,s) (r^2-1)^s  \frac{\pi^{n/2}}{\Gamma(n/2)}\int_r^{\infty} \frac {2\rho}{(\rho^2-r^2)^s(\rho^2-1)} \, d\rho. \]
We change variable $ \frac{\rho^2-r^2}{r^2-1}=z$ and obtain
	\begin{equation*}
		\begin{split}
                              \int_{\obal}P_r(y,x) \, dy
                               	=   c(n,s)   \frac{\pi^{n/2}}{\Gamma(n/2)} \int_0^{\infty} \frac {1}{z^s \, (z+1)}\, dz.
		\end{split}
	\end{equation*}
We use \eqref{betappl} and the value of $c(n,s)$ from \eqref{ctcns} and obtain that
	\begin{equation*} \int_{\obal}P_r(y,x) \, dy= 1. \end{equation*}
This completes the proof of Lemma \ref{Ip}.
\end{proof}

\begin{lemma}
For any $r>0$ and any $x \in B_r $
	\begin{equation}  c(n,s) \int_{B_r} \frac {dy}{(r^2-|y|^2)^s |x-y|^{n-2s}} =1 \label{If} .\end{equation}
\end{lemma}

\begin{proof}
Let $y^*$ be the inversion of $y$ with center at $x$ (notice that $ |y^*|>r$). Then by using \eqref{firsttr}
and \eqref{dxtr}
we obtain that
	\begin{equation*}\int_{B_r} \frac {dy}{(r^2-|y|^2)^s |x-y|^{n-2s}}=\int_{\obal} \Bigg(\frac {r^2-|x|^2}{|y^*|^2-r^2}\Bigg)^s \frac{dy^*}{|x-y^*|^n}.
	\end{equation*}
From identity \eqref{Ip} the desired result immediately follows.
\end{proof}

\begin{lemma}
For any $r>0$ and any $x \in \obal$
\begin{equation}\int_{\obal}A_r(y) \Phi(x-y)\, dy = \Phi(x)\label{Ifu}. \end{equation}
\end{lemma}

\begin{proof}
We prove the claim for $n\neq 2s$. We insert definitions \eqref{smeandefn} and \eqref{fundsolution} and obtain that
 \begin{equation*}
	\begin{split}
			\int_{\obal}A_r(y) \Phi(x-y)\, dy  =  r^{2s}c(n,s) a(n,s)\int_{\obal} \frac{1}{(|y|^2-r^2)^s|y|^n |x-y|^{n-2s}}\, dy.
	\end{split}
\end{equation*}
Let $x^*$ and $y^*$ be the inversion of $x$, respectively $y$ with center at $0$. Using identities \eqref{tr},
\eqref{firsttr}
 \eqref{sectr}
and \eqref{dxtr}
we obtain that	 \begin{equation*}
		\begin{split}
			\int_{\obal}A_r(y) \Phi(x-y)\, dy
			= \frac{c(n,s)   a(n,s)}{|x|^{n-2s}}\int_{B_r} \frac{dy^*}{|x^*-y^*|^{n-2s}\,  \big( r^2-|y^*|^2 \big)^s}.
				\end{split}
	\end{equation*}
From \eqref{If} it follows that
 	\begin{equation*}
			\int_{\obal}A_r(y) \Phi(x-y)\, dy =  \frac{a(n,s)}{|x|^{n-2s}},
	\end{equation*}
and thus the desired result.

We now prove the claim for $n=2s$, assuming $r=1$. We have that
\[ \int_{\obal}A_r(y) \Phi(x-y)\, dy =  \frac{-1}{\pi^2}\int_{|y|>1} \frac {\log|y-x|}{ \sqrt{ y^2-1} |y|} dy.\]
We perform the change of variables $v= \frac{1}{y}$, with $|v|\leq 1$. We set $w:= \frac{1}{x}$, hence $|w|\leq 1$.
Then we have that
	\[ \int_{\obal}A_r(y) \Phi(x-y)\, dy   = \frac{-1}{\pi^2} \int_{|v|\leq 1}\bigg( \log \frac{|v-w|}{|v|} + \log|x|\bigg)\frac{dv}{ \sqrt{1-v^2}}.\]
We use the following result (see \cite{conto}, page 549)
\begin{equation}
	   \int_{|v|\leq 1} \log |v-a| \frac{dv}{ \sqrt{1-v^2}}=  \begin{cases}  -\pi \log 2, &\text{ if } |a|\leq 1\\
								 \pi \log (|a|+(a^2-1)^{1/2}) - \pi \log 2, &\text{ if } |a|\geq 1.
	\end{cases}
	\label{logid}
	\end{equation}
We thus obtain \[\int_{\obal}A_r(y) \Phi(x-y)\, dy = -\frac{1}\pi\log|x|,\]
which concludes the proof of the Lemma.
\end{proof}

\begin{lemma}
For any $r>0$, let $x_0 \in B_r$ be a fixed point.  For any $x \in \obal$
	 \begin{equation} \int_{\obal}P_r(y,x_0) \Phi(x-y) \, dy  = \Phi(x-x_0). \label{Ipu} \end{equation}
\end{lemma}

\begin{proof}
We prove the claim for $n\neq 2s$.
We have that  \begin{equation*}
	\begin{split}
		\int_{\obal}P_r(y,x_0) \Phi(x-y) \, dy = 	c(n,s) a(n,s) \int_{\obal} \frac{(r^2-|x_0|^2)^s|x-y|^{2s-n}\, dy} {(|y|^2-r^2)^s|y-x_0|^n} .
	\end{split}
\end{equation*}
Let $x^*$ and $y^*$ be the inversion of $x$, respectively $y$ with center at $x_0$. From \eqref{tr},
 \eqref{firsttr}
\eqref{dxtr} and
\eqref{sectr}
 we have that
	\begin{equation*}\begin{split}
				&  \int_{\obal}P_r(y,x_0) \Phi(x-y) \, dy \\ =&\; c(n,s) a(n,s) \frac{|x^*-x_0|^{n-2s}  }{(r^2-|x_0|^2)^{n-2s}  } \int_{B_r}  \frac {dy^*}{(r^2-|y^*|^2)^s|y^*-x^*|^{n-2s}}.
		\end{split} \end{equation*}
Using \eqref{If}, we obtain that \begin{equation*} \int_{\obal}P_r(y,x_0) \Phi(x-y) \, dy =  \frac{a(n,s)}{|x-x_0|^{n-2s}},\end{equation*}
which concludes the proof for $n\neq 2s$.

We now prove the claim for $n=2s$, assuming $r=1$. We have that
\[ \int_{\obal}P_r(y,x_0) \Phi(x-y) \, dy = \frac{-1}{\pi^2} \int_{|y|>1} \sqrt{ \frac{1-{x_0}^2}{y^2-1}}  \frac{\log|y-x|}{|y-x_0|} dy.\]
We perform the change of variables $v=\frac{yx_0-1}{y-x_0}$, noticing that $|v|\leq 1$. We set $w:= \frac{x x_0-1}{x-x_0}$, hence $|w|\leq 1$.
Then we have that
	\[\int_{\obal}P_r(y,x_0) \Phi(x-y) \, dy =\frac{-1}{\pi^2} \int_{|v|\leq 1} \bigg( \log  \frac{|v-w|}{|v-x_0|}+ \log |x-x_0|\bigg) \frac{ dv}{\sqrt{1-v^2}}.\]
We use identity \eqref{logid} and obtain \[\int_{\obal}P_r(y,x_0) \Phi(x-y) \, dy = -\frac{1}{\pi}\log|x-x_0|,\]
which concludes the proof.
\end{proof}
 We emphasize here two computations that we used in the proof of Lemma \ref{lemmaip}, namely identities \eqref{aahhh} and \eqref{sinstuff}.

\begin{proposition} For any $\tau >1$
	\begin{equation} \label{prop1}  \int_0^{\pi} \frac {\sin^{n-2}\theta}{(\tau^2-2\tau \cos\theta +1)^{n/2}} \, d\theta = \frac {1} {\tau^{n-2} (\tau^2-1)}\int_0^{\pi} \sin^{n-2} \alpha \, d\alpha .\end{equation}
\end{proposition}

\begin{proposition}
\begin{equation} \pi \prod_{k=1}^{n-2} \int_0^{\pi}\sin^k\theta \, d\theta = \frac{{\pi}^{n/2}}{\Gamma (n/2)}.\label{prop2} \end{equation}
\end{proposition}
In the next Proposition we introduce yet another integral identity.
\begin{proposition}
\label{proposition:uss} Let $\alpha, \beta \in \mathbb{R}$ such that $ \bigg| \frac{\alpha}{\alpha +\beta}\bigg| < 1$. Then
\[  \int_0^\alpha \frac{(\alpha-x)^{s-1}}{x^s(\beta+x)}\, dx  = \frac{\pi}{\sin(\pi s)} \frac{(\alpha+\beta)^{s-1}} {\beta^s}.\]
\end{proposition}

\begin{proof}

We change the variable $x =\alpha t$ and obtain that
	\begin{equation*}
		 \int_0^\alpha \frac{(\alpha-x)^{s-1}}{x^s(\beta+x)}\, dx = \frac{1}{\beta}  \int_0^1 t^{-s} (1-t)^{s-1} \bigg(1+ \frac{\alpha}{\beta}t\bigg)^{-1}\, dt.
	\end{equation*}
We use the integral definition \eqref{inthyp} of the hypergeometric function for $a=1$, $b=1-s$, $c=1$ and $w= -\frac{\alpha}{\beta}$ (since $|t|<1$ the integral is convergent) and we obtain that
	\[  \int_0^1 t^{-s} (1-t)^{-s} \bigg(1+ \frac{\alpha}{\beta}t\bigg)^{-1}\, dz =  \frac{\Gamma(s) \Gamma(1-s)}{\Gamma(1)} F\bigg(1,1-s,1,-\frac{\alpha}{\beta}\bigg).\]
Now we use the linear transformation \eqref{hyp3} and compute
	\[ F\bigg(1,1-s,1,-\frac{\alpha}{\beta}\bigg) = \bigg(\frac{\alpha+\beta}{\beta} \bigg)^{s-1} F\bigg(1-s,0,1,\frac{\alpha}{\alpha+\beta}\bigg)  .\]
We use the Gauss expansion in \eqref{gausshyp} and notice that for $k>0$, all the terms of the sum vanish. We are left with only with the term $k=0$ and obtain that
	\[ F\bigg(1-s,0,1,\frac{\alpha}{\alpha+\beta}\bigg)=1.\]
Furthermore, $\Gamma(1-s)\Gamma(s)= \frac{\pi}{\sin (\pi s)}$ (as in identity \eqref{gam3}) and it follows that
	\begin{equation*}\int_0^\alpha \frac{(\alpha-x)^{s-1}}{x^s(\beta+x)}\, dx=  \frac{\pi}{\sin \pi s} \frac{(\alpha+\beta)^{s-1}}{\beta^s}.\qedhere \end{equation*}
\end{proof}

We explicitly compute here another integral that was used in our computations, namely :

\begin{proposition}For any $s\in (0,1/2]$ we have that
\begin{equation} \label{ctcomp1111} \int_0^\infty t^{2s-2} \sin t \, dt =-\cos(\pi s) \Gamma(2s-1).\end{equation}
\end{proposition}
\begin{proof}
We have that
\begin{equation}\label{sti1} \int_0^{\infty} t^{2s-2} \sin t \, dt= - \text{Im} \int_0^{\infty}  t^{2s-2} e^{-it}\, dt. \end{equation}
We consider the closed path $\Omega_{\rho} =\partial\Big( \big([0,\rho]\times [0,\rho] \big) \cap B_\rho(0) \Big)$. We take the contour integral $\int_{\Omega_{\rho}} z^{2s-2} e^{-z}\, dz $, and let $\gamma_{\rho} = \partial  B_\rho(0) \cap  \big([0,\rho]\times [0,\rho] \big) $ (the boundary of the quarter of the circle). By Cauchy's Theorem, the contour integral is 0 (there are no poles inside $\Omega_{\rho}$), therefore
\[ \int_0^{\rho} t^{2s-2}e^{-t} \, dt + \int_{\gamma_{\rho}} z^{2s-2} e^{-z}\, dz - i \int_0^{\rho} (it)^{2s-2} e^{-it}\, dt =0.\] Integrating along $\gamma_{\rho}$, by using polar coordinates $z={\rho}e^{i\theta}$ and then the change of variables $\cos \theta =t$ we have that
	\[ \begin{split} \bigg| \int_{\gamma_{\rho}} z^{2s-2} e^{-z} \, dz \bigg| =&\; \bigg| \int_0^{\pi/2} {\rho}^{2s-1} e^{i\theta (2s-1)} e^{-{\rho}e^{i\theta}} \, d\theta\bigg|
	\leq {\rho}^{2s-1} \bigg| \int_0^{\pi/2} e^{-{\rho}\cos \theta}\, d\theta \bigg| \\
	 				=&\; {\rho}^{2s-1} \bigg| \int_0^1 \frac{ e^{-{\rho}t}}{\sqrt{1-t^2} }\, dt \bigg| \\
	 				\leq&\; {\rho}^{2s-1} e^{-{\rho}/2}  \bigg| \int_{1/2}^1 (1-t)^{-1/2}\, dt \bigg|   + \overline  c {\rho}^{2s-1}  \bigg|  \int_0^{1/2}  e^{-{\rho}t} \, dt \bigg| \\
	 				=&\; c  {\rho}^{2s-1} e^{-{\rho}/2} + \overline c {\rho}^{2s-2} (e^{-{\rho}/2}-1) .
\end{split}\]  Hence \[ \lim_{{\rho}\to \infty} \int_{\gamma_{\rho}} z^{2s-2} e^{-z} \, dz =0\]
and we are left only with the integrals along the real and the imaginary axis, namely
	\[ \int_0^{\infty} t^{2s-2}e^{-t} \, dt = i^{2s-1} \int_0^{\infty} t^{2s-2} e^{-it} \, dt .\] Here the left hand side returns the Gamma function according to definition \eqref{gamma}. We compute $i^{1-2s} =\big( \cos (\pi/2) + i\sin(\pi /2)\big) ^{1-2s}= \sin(\pi s)+ i\cos(\pi s)$ and in \eqref{sti1} we obtain that
		 \[   \int_0^{\infty} t^{2s-2} \sin t \, dt= - \Gamma(2s-1)  \text{Im} \Big(   \sin(\pi s)+ i\cos(\pi s)\Big) = -\cos(\pi s) \Gamma(2s-1).\]
		 This concludes the proof of the Proposition.
\end{proof}

\section*{Acknowledgments} {I am greatly indebted to Professor Enrico Valdinoci for his patience, his guidance and his precious help. I sincerely thank Matteo Cozzi for his appropriate and very useful observations. A great thanks also to Milosz Krupski, for his attentive reading and very nice remarks.}

\medskip
Received xxxx 20xx; revised xxxx 20xx.
\medskip

\end{document}